\documentclass[11pt,a4paper,%
               BCOR20mm,DIV16,%
               headings=normal,%
               twoside,%
               headtopline,%
               headsepline,%
               footsepline,%
               plainfootsepline,%
               abstracton,%
               numbers=enddot,%
               footinclude=false]{scrartcl}
\usepackage{scrpage2}
\usepackage{amsmath,amssymb,amsthm}
\usepackage{charter}%
\swapnumbers
\theoremstyle{plain}
\newtheorem{theo}{Theorem}[section]
\newtheorem{prop}[theo]{Proposition}
\newtheorem{coro}[theo]{Corollary}
\newtheorem{lemm}[theo]{Lemma}
\newtheorem{namet}[theo]{\myThmName}
\newenvironment{nthm}[1][]{\edef\myThmName{#1}\begin{namet}}{\end{namet}}
\theoremstyle{definition}
\newtheorem{exam}[theo]{Example}
\newtheorem{exas}[theo]{Examples}
\newtheorem{prob}[theo]{Problem}
\newtheorem{rema}[theo]{Remark}
\newtheorem{rems}[theo]{Remarks}

\newtheorem{defi}[theo]{Definition}
\newtheorem{defs}[theo]{Definitions}
\newtheorem{nota}[theo]{Notation}

\newtheorem{acks}[theo]{Acknowledgements}
\newtheorem{named}[theo]{\myThmName}
\newtheorem*{named*}{\myThmName}
\newenvironment{ndef}[1][]{\edef\myThmName{#1}\begin{named}}{\end{named}}
\newenvironment{ndef*}[1][\kern-.35em]{\edef\myThmName{#1}\begin{named*}}{\end{named*}}
\rehead{\small\rm\MYauthor}
\lehead{\small\rm\thepage}%
\lohead{\footnotesize\rm
    \def\\{}
    \footnotesize\MYtitle}%
\rohead{\small\rm\thepage}%
\lefoot{}
\rofoot{}
\setheadtopline{.4pt} 
\setheadsepline{.8pt}
\setfootsepline{.4pt}

\newcommand{\CC}{\mathbb C}
\newcommand{\FF}{\mathbb F}
\newcommand{\HH}{\mathbb H}
\newcommand{\KK}{\mathbb K}
\newcommand{\LL}{\mathbb L}
\newcommand{\MM}{\mathbb M}
\newcommand{\QQ}{\mathbb Q}
\newcommand{\PP}{\mathbb P}
\newcommand{\RR}{\mathbb R}
\newcommand{\ZZ}{\mathbb Z}
\newcommand{\pfaff}{\mathrm{pf}}
\newcommand{\id}{\mathrm{id}}
\newcommand{\disc}{\operatorname{disc}}
\newcommand{\GL}[2]{\def\test{#1}\ifx\empty\test
      \mathrm{GL}(#2)\else\mathrm{GL}_{#1}#2\fi}
\newcommand{\SL}[2]{\def\test{#1}\ifx\empty\test
      \mathrm{SL}(#2)\else\mathrm{SL}_{#1}#2\fi}
\newcommand{\PSL}[2]{\def\test{#1}\ifx\empty\test
      \mathrm{PSL}(#2)\else\mathrm{PSL}_{#1}#2\fi}
\newcommand{\PGL}[2]{\def\test{#1}\ifx\empty\test
      \mathrm{PGL}(#2)\else\mathrm{PGL}_{#1}#2\fi}
\newcommand{\Sp}[2]{\def\test{#1}\ifx\empty\test
      \mathrm{Sp}(#2)\else\mathrm{Sp}_{#1}#2\fi}
\newcommand{\GSp}[2]{\def\test{#1}\ifx\empty\test
      \mathrm{GSp}(#2)\else\mathrm{GSp}_{#1}#2\fi}
\newcommand{\GO}[2]{\def\test{#1}\ifx\empty\test
      \mathrm{GO}(#2)\else\mathrm{GO}_{#1}#2\fi}
\newcommand{\GOplus}[2]{\def\test{#1}\ifx\empty\test
      \mathrm{GO}^+(#2)\else\mathrm{GO}^+_{#1}#2\fi}
\newcommand{\SO}[2]{\def\test{#1}\ifx\empty\test
      \mathrm{SO}(#2)\else\mathrm{SO}_{#1}#2\fi}
\newcommand{\Orth}[2]{\def\test{#1}\ifx\empty\test
      \mathrm{O}(#2)\else\mathrm{O}_{#1}#2\fi}
\newcommand{\PGO}[2]{\def\test{#1}\ifx\empty\test
      \mathrm{PGO}(#2)\else\mathrm{PGO}_{#1}#2\fi}
\newcommand{\PGOplus}[2]{\def\test{#1}\ifx\empty\test
      \mathrm{PGO}^+(#2)\else\mathrm{PGO}^+_{#1}#2\fi}
\newcommand{\Aut}[2][]{\mathrm{Aut}_{#1}(#2)}
\newcommand{\End}[2][]{\mathrm{End}_{#1}(#2)}

\newcommand{\Hom}[3][]{\mathrm{Hom}_{#1}(#2,#3)}
\newcommand{\Gal}[1]{\mathrm{Gal}(#1)}
\newcommand{\G}[2]{\mathrm{Gr}_{#1,#2}^\wedge}
\newcommand{\Sk}[1]{\KK^{#1}\wedge \KK^{#1}}
\newcommand{\s}[2]{S^{#1}_{#2}}
\newcommand{\M}[2]{\KK^{#1\times#2}}
\newcommand{\opp}{\rho}
\newcommand{\actSL}{\#}
\newcommand{\frob}{\varphi}
\let\setminus\smallsetminus
\makeatletter
\newcommand{\widebar}[1]{ {\mathchoice
    {\vbox
    {\m@th \ialign {##\crcr \noalign {\kern 1\p@ }\kern 1\p@ \hrulefill \crcr
        \noalign {\kern 1\p@ \nointerlineskip }%
        $\hfil \displaystyle {#1}\hfil $\crcr }}}
    {\vbox
    {\m@th \ialign {##\crcr \noalign {\kern 1\p@ }\kern 1\p@ \hrulefill \crcr
        \noalign {\kern 1\p@ \nointerlineskip }%
        $\hfil \textstyle {#1} $\crcr }}}
    {\vbox
    {\m@th \ialign {##\crcr \noalign {\kern 1\p@ }\kern 1\p@ \hrulefill \crcr
        \noalign {\kern 1\p@ \nointerlineskip }%
        $\hfil \scriptstyle {#1}\hfil $\crcr }}}
    {\vbox
    {\m@th \ialign {##\crcr \noalign {\kern 1\p@ }\kern 1\p@ \hrulefill \crcr
        \noalign {\kern 1\p@ \nointerlineskip }%
        $\hfil \scriptscriptstyle {#1}\hfil $\crcr }}}%
    }}
\makeatother
\let\gal\widebar
\makeatletter
\newcommand{\myBox}[1]{{\mathchoice
    {\vbox
    {\m@th \ialign {##\crcr \hline \crcr
        \noalign {\nointerlineskip }%
        \vrule
        \vphantom{x}\smash{\lower.21\p@\hbox{$\kern-1.7\p@ \displaystyle {#1}\kern-1.7\p@$}}\vrule\crcr
        \noalign {\nointerlineskip  }\crcr\hline%
      }}}
    {\vbox
    {\m@th \ialign {##\crcr \hline \crcr
        \noalign {\nointerlineskip }%
        \vrule
        \vphantom{x}\smash{\lower.25\p@\hbox{$\kern-1.7\p@ \textstyle {#1}\kern-1.7\p@$}}\vrule\crcr
        \noalign {\nointerlineskip  }\crcr\hline%
      }}}
    {\vbox
    {\m@th \ialign {##\crcr \hline \crcr
        \noalign {\nointerlineskip }%
        \vrule
        $\scriptstyle\vphantom{x}$\smash{\lower.25\p@\hbox{$\kern-1.7\p@ \scriptstyle {#1}\kern-1.6\p@$}}\vrule\crcr
        \noalign {\nointerlineskip  }\crcr\hline%
      }}}
    {\vbox
    {\m@th \ialign {##\crcr \hline \crcr
        \noalign {\nointerlineskip }%
        \vrule
        $\scriptscriptstyle\vphantom{x}$\smash{\lower.18\p@\hbox{$\kern-1.6\p@ \scriptscriptstyle {#1}\kern-1.5\p@$}}\vrule\crcr
        \noalign {\nointerlineskip  }\crcr\hline%
      }}}%
    }}
\makeatother
\newcommand{\myTimes}{\vphantom{x}\smash{\times}}
\newcommand{\sq}{\myBox{\phantom{\myTimes}}}
\newcommand{\sqt}{\myBox{\myTimes}}
\newcommand{\set}[2]{\left\{{#1}\left|\vphantom{#1}\vphantom{#2}\right.\, 
                     \strut{#2}\right\}}
\newcommand{\smallset}[2]{\left\{\smash{#1}\left|\right.\, 
                     \smash{#2}\right\}}
\newcommand{\smallspnK}[1]{\langle#1\rangle^{}_\KK}
\newcommand{\spnK}[1]{\left\langle#1\right\rangle^{}_\KK}
\newcommand{\smallspnKperp}[1]{\langle#1\rangle^\perp_\KK}
\newcommand{\spn}[2]{\left\langle#1\right\rangle^{}_{#2}}
\newcommand{\smallspn}[2]{\langle#1\rangle^{}_{#2}}
\newcommand{\Char}{\mathop{\mathrm{char}}}
\newcommand{\GHeis}[3]{\mathrm{GH}(#1,#2,#3)}
\newcommand{\gheis}[4][]{\mathfrak{gh}_{#1}(#2,#3,#4)}

\newcommand{\z}[1]{\mathfrak{z}(#1)}
\newcommand{\g}{\mathfrak{g}}
\newcommand{\formspace}[3]{P_{#1,#2}^{#3}}
\newcommand{\Formspace}[4]{P_{#1,#2,#3}^{#4}}
\newcommand{\N}{N}
\newcommand{\Nperp}{N^\perp}
\newcommand{\Pu}[1]{\mathrm{Pu}(#1)}
\newcommand{\tr}{\mathop{\mathrm{tr}}}
\usepackage{mdwlist}
\usepackage{paralist}%

\newenvironment{enum}{\begin{compactenum}}{\end{compactenum}}

\newcounter{rememberEnumi}

\newcommand{\keywords}[1]{\bgroup
\renewcommand{\thefootnote}{}\footnote{{\bf Key words: }#1} 
\egroup\setcounter{footnote}{0}}
\newcommand{\classification}[1]{\bgroup\renewcommand{\thefootnote}{}%
\footnote{{\bf MSC 2000 (Mathematics Subject Classification): }#1}
\egroup\setcounter{footnote}{0}}
\addtokomafont{sectioning}{\normalfont\rmfamily\bfseries}
\title{\vspace*{-10mm}%
  Stabilizers of Subspaces\\ under Similitudes of the Klein Quadric, 
       \\and Automorphisms of Heisenberg Algebras} 
\author{Michael Gulde, Markus Stroppel}
\newcommand{\MYtitle}{Stabilizers of Subspaces, and Automorphisms of Heisenberg Algebras}
\newcommand{\MYauthor}{Michael Gulde, Markus Stroppel}
\date{\vspace{-\baselineskip}}
\RequirePackage{hyperref} 
\begin{document}
\maketitle
\classification{17B30, 
                22E25, 
                15A63, 
                15A69, 
                15A72, 
                51A50, 
}
\begin{abstract}\noindent
  We determine the groups of automorphisms and their orbits for nilpotent
  Lie algebras of class~$2$ and small dimension, over arbitrary fields
  (including the characteristic~$2$ case).   
\end{abstract}

\section*{Introduction}

In~\cite{MR2431124}, a conceptual approach has been given to
the classification of Lie algebras $L$ with $L':=[L,L]\le\z L$ for
small values of $\dim(L/\z L)$ and over an arbitrary field~$\KK$ (including
the case~$\Char\KK=2$).  In the present paper, we use this to
find all automorphisms of Lie algebras $L$ with $L'\le\z L$ and
$\dim(L/L')=4$, 
clarify the structure of the group $\Aut L$, and determine its orbits
on~$L$. 

The restriction to the case $\dim(L/L')=4$ is justified by the fact
that the classification problem becomes wild for $\dim(L/L')>4$, even
if we assume that the ground field is algebraically closed,
cf.~\cite{MR2161930}, \cite{MR0491938}. 

One motivation for the study of Lie algebras over arbitrary fields
comes from group theory.  If one wants to understand nilpotent groups,
Lie methods are useful in many cases. For instance, we have
(cf.~\cite[{VIII\,9.16}]{MR650245}): If $G$ is a nilpotent group of
class at most~$2$ in which every element has a unique square root then
$x+y:=xy\sqrt{[y,x]}$ defines the addition of a Lie ring (i.e., a Lie
algebra over the ring~$\ZZ$ of integers) $\g=(G,+,[\cdot,\cdot])$, and
$\Aut{G}=\Aut{\g}$. Moreover, this addition coincides with the
multiplication on any cyclic subgroup of~$G$.  If the group $G$ has
exponent~$p\in\PP$ (where $p\ne2$ by our assumption on square roots)
then the Lie ring may also be considered as a Lie algebra over the
field with~$p$ elements.
Our present work thus comprises a generalization of the results
in~\cite{MR593560}. See~\cite{MR795568} for extensions in a more
group-theoretic manner, and~\cite{MR2410562} for an investigation of a
similar class of nilpotent groups. 

\goodbreak
\begin{ndef*}[Organization of the paper]\quad
  \begin{description*}
  \item[] Section~\ref{sec:autHeis}: %
    {Automorphisms of Heisenberg algebras} 
  \item[] Section~\ref{sec:toolsForms}: %
    {Tools from the classification of forms} %
  \item[] Section~\ref{sec:classHeis}: %
    {Classification of reduced Heisenberg algebras}
  \item[] Section~\ref{sec:directComputations}: %
    {The cases where $\Sigma_\beta$ can be computed directly}
  \item[] Section~\ref{sec:fieldextensions}: %
    {Examples involving field extensions}
  \item[] Section~\ref{sec:quaternions}: %
    {Examples involving quaternion algebras}
  \item[] Section~\ref{sec:results}: %
      {Results}
  \end{description*}
\end{ndef*}

\goodbreak
\section{Automorphisms of Heisenberg algebras}\label{sec:autHeis}
\begin{defs}
  A \emph{generalized Heisenberg algebra}
  $\gheis VZ\beta:=(V\times Z,[\cdot,\cdot]_\beta)$ is given by vector
  spaces $V$ and $Z$ together with an alternating bilinear map
  $\beta\colon V\times V\to Z$:
  the underlying vector space is $V\times Z$, and the Lie bracket is 
  $[(v,x),(w,y)]_\beta:=(0,\beta(v,w))$.

  If the image $\beta(V\times V)$ generates~$Z$, and
  $\smallset{v\in V}{\forall\, w\in V\colon \beta(v,w)=0}=\{0\}$,
  we call $\gheis VZ\beta$ a \emph{reduced Heisenberg algebra}.
  These conditions mean that $\{0\}\times Z$ equals both the center
  and the commutator algebra of~$\gheis VZ\beta$. 

  Using the universal property of the tensor product, we obtain a
  unique linear surjection $\hat\beta\colon V\wedge V\to Z$ such that\/ 
  $\hat\beta(v\wedge w)=\beta(v,w)$ holds for all\/ $v,w\in V$.
  If $\gheis VZ\beta$ is reduced then the kernel of\/ $\hat\beta$
  satisfies 
\[
  \forall\, v\in V\setminus\{0\}\colon 
  \eta(\{v\}\times V)\not\subseteq\ker\hat\beta 
\eqno{(\ast)}
\]
  where $\eta(v,w)=v\wedge w$. 
  Conversely, every linear surjection $\gamma\colon V\wedge V\to Z$
  satisfying condition~{$(\ast)$} yields a reduced Heisenberg algebra
  $\gheis VZ{\gamma\circ\eta})$. 
\end{defs}

  Every nilpotent Lie algebra of class~$2$ is isomorphic to the direct
  sum of a reduced Heisenberg algebra and an abelian Lie algebra,
  cf.~\cite[6.2]{MR2431124}.  

\begin{prop}\label{isoHeis}
  Let\/ $\gheis VZ\beta$ be a reduced Heisenberg algebra, and let $A$
  be an abelian Lie algebra.
  Let\/ $\Sigma_\beta$ denote the group of all linear bijections of $V$ onto
  itself such that\/ $\sigma(\ker\hat\beta)=\ker\hat\beta$.
  For $\sigma\in\Sigma_\beta$, let $\sigma'$ denote the unique element of\/
  $\GL{}Z$ such that
\[
  \forall\, u,v\in V\colon  \sigma'(\beta(u,v))=\beta(\sigma(u),\sigma(v)) \,.
\eqno{(\ast\ast)}
\]
  Then the automorphisms of the Lie algebra $\gheis VZ\beta\times A$
  are precisely the maps
  \[
  (v,z,a)\mapsto\left( \sigma(v) , \sigma'(z)+\tau(v)+\zeta(a) ,
    \alpha(a)+\xi(v) \right) \,,
  \]
  where $\sigma\in\Sigma_\beta$, $\tau\in\Hom VZ$, $\zeta\in\Hom AZ$, $\alpha\in\GL{}A$, and
  $\xi\in\Hom VA$. 
\end{prop}
\begin{proof}
  An easy computation shows that the given maps are automorphisms. 
  The center of $\gheis VZ\beta\times A$ is $\{0\}\times Z\times A$,
  its commutator algebra is $\{0\}\times Z\times\{0\}$.
  Since these are invariant subalgebras, there are no other
  automorphisms than the ones we have described. 
\end{proof}

For the case of reduced algebras this result has already been stated
in~\cite[4.4]{MR1484565}.

\section{Tools from the classification of forms}\label{sec:toolsForms}

The classification of Heisenberg algebras  $\gheis VZ\beta$ with $\dim
V=4$ and our determination of the corresponding groups of automorphisms
and their orbits is based on the study of restrictions of the Pfaffian
form, see~\ref{O3Kdef} below. 

Aiming at applications later on, we present well known results here,
mainly on forms in at most four variables. For the reader's
convenience we indicate proofs and include a discussion of quaternions
over arbitrary fields.

\begin{ndef}[The Arf Invariant]\label{Arf}
  When dealing with quadratic forms over a field~$\KK$ with
  $\Char\KK=2$ one has to distinguish between diagonalizable quadratic
  forms and non-diagonalizable ones. Equivalence of the latter depends
  on Arf's invariant~$\delta$ (cf.~\cite[8.11]{MR1935285}): for a 
  non-diagonalizable form $q$ in two variables given by $q(v):=v'Mv$
  this is just
  $\delta(q) := \frac{\det M}{(\tr (iM))^2} + \wp$,
  where $\wp:=\set{x+x^2}{x\in\KK}$.
  If we describe the same form by different matrices $M$ and
  $\tilde{M}$ then $M-\tilde{M} = ti$ for some $t\in\KK$ and $i=\left(
    \begin{smallmatrix}
      0 & 1 \\
      -1 & 0 
    \end{smallmatrix}\right)$. Thus 
  $\frac{\det M}{(\tr(iM))^2} -
  \frac{\det\tilde{M}}{(\tr(i\tilde{M}))^2} =
  \frac{t}{\tr(iM)}+\bigl(\frac{t}{\tr(iM)}\bigr)^2 \in\wp$ shows that
  $\delta$ is indeed an invariant of the form~$q$. 
  
  It is easy to see that $\delta(q)$ does not change if we replace $q$
  by $tq$ with $t\in\KK^\times$, or if we pass to an equivalent form
  (replacing $M$ by $A'MA$).
\end{ndef}

In order to understand the set of values of a non-diagonalizable form
we multiply the form with a scalar such that it assumes the value~$1$
at some $v\in\KK^2$. Then the form is equivalent to the form~$q$ given
by $q(v)=v'\left(
  \begin{smallmatrix}
    1 & t \\
    0 & d 
  \end{smallmatrix}\right)v$,
and may be interpreted as the (multiplicative) norm form of a suitable
algebra, as follows.

We write $u$ for the class of $X$ modulo $(X^2+tX+d)$ and consider the
algebra $\KK(u):=\KK[X]/(X^2+tX+d)$. Putting $\gal{a+bu}:=a+tb-bu$
we obtain an involutory automorphism interchanging the two roots $u$
and $t-u$ of $X^2+tX+d$ in $\KK(u)$; the corresponding norm is
$(\gal{a+bu})(a+bu) = q(a,b)$. Clearly this norm is
multiplicative. Moreover, a product in $\KK(u)$ has norm~$0$ only if
one of the factors has norm~$0$, and $\KK(u)^\times
=\set{a+bu\in\KK(u)}{q(a,b)\ne0}$ is the group of units in~$\KK(u)$.
Thus $V_{1,d}:=
\set{a^2+abt+b^2d}{a,b\in\KK}\setminus\{0\} =
\set{\gal{v}v}{v\in\KK(u)^\times}$ is a subgroup of the multiplicative
group~$\KK^\times$.

\begin{nthm}[Equivalence of Non-diagonalizable Forms in Characteristic Two]
\label{nonDiagonalCharTwo}
  Let $q$ and $\tilde q$ be non-diagonalizable forms in two variables
  over a field\/~$\KK$ of characteristic~$2$. Then $q$ and $\tilde{q}$ are
  equivalent precisely if their Arf invariants are equal and the forms
  share a non-zero value, i.e., there exist non-zero vectors $v,w$
  such that $q(v)=\tilde{q}(w)$.

  A set of representatives for the equivalence classes of
  non-diagonalizable forms is obtained by taking the forms described
  by the matrices in $
  \left\{\left(
  \begin{smallmatrix}
    0 & 0 \vphantom{1}\\
    0 & 0 \vphantom{d}
  \end{smallmatrix}\right)\right\}
  \cup 
  \set{a\left(
  \begin{smallmatrix}
    1 & 1 \\
    0 & d
  \end{smallmatrix}\right)}{d\in R_\wp, a\in R_d}$
  where $R_\wp$ and $R_d$ are sets of representatives of the cosets in
  $\KK/\wp$ and $\KK^\times/V_{1,d}$, respectively.  

\end{nthm}
\begin{proof}
  We already know from~\ref{Arf} that $\delta(q)$ and
  $\delta(\tilde{q})$ coincide if $q$ and~$\tilde{q}$ are
  equivalent. Clearly, equivalent forms share a value (indeed, they
  have the same range).

  Now assume, conversely, that $q$ and~$\tilde{q}$ have the same Arf
  invariant and that they share a non-zero value. Upon basis
  transformation we may assume that $0\ne a:=q(1,0)=\tilde{q}(1,0)$. 
  Then the forms are represented by matrices $a\left(
    \begin{smallmatrix}
      1 & x \\
      0 & c
    \end{smallmatrix}\right)$ and $a\left(
    \begin{smallmatrix}
      1 & w \\
      0 & d
    \end{smallmatrix}\right)$, %
  respectively, with $x\ne0\ne w$ because the forms are not
  diagonalizable. Our assumption $\delta(q)=\delta(\tilde{q})$ yields
  the existence of $k\in\KK$ such that $\frac{c}{x^2}+\frac{d}{w^2}
  = k^2+k$.  Computing
  \[
  \left(
    \begin{matrix}
      1 & 0 \\
      kw &  w/x
    \end{matrix}\right)
  \left(
    \begin{matrix}
      1 & x \\
      0 & c
    \end{matrix}\right)
  \left(
    \begin{matrix}
      1 & kw \\
      0 &  w/x
    \end{matrix}\right)
  = 
  \left(
    \begin{matrix}
      1 & kw+w \\
      kw & w^2(k^2+k+c/x^2)
    \end{matrix}\right)
  = 
  \left(
    \begin{matrix}
      1 & kw+w \\
      kw & d
    \end{matrix}\right)
   \]
  we see that the forms are equivalent.    

  In order to find the representatives we first choose a
  representative $d\in R_\wp$ for the Arf invariant of a given
  form. Then the form is equivalent to the form described by $k\left(
    \begin{smallmatrix}
      1 & 1 \\
      0 & d
    \end{smallmatrix}\right)$ for any value $k\in\KK^\times$ assumed
  by the form. It thus remains to choose a representative $a$ for the
  coset $kV_{1,d}$.
\end{proof}

 \begin{ndef}[Diagonal Forms in Characteristic Two]\label{diagonalCharTwo}
   We continue to discuss forms in two variables over a field $\KK$ with
   $\Char\KK=2$. If such a form is described by a diagonal matrix $\left(
     \begin{smallmatrix}
       x & 0 \\
       0 & z
     \end{smallmatrix}\right)$
   a change of basis will have the effect of changing this matrix to
       \[
    \left(
      \begin{matrix}
        a & b \\
        c & d 
      \end{matrix}\right) 
    \left(
      \begin{matrix}
        x & 0 \\
        0 & z
      \end{matrix}\right)
    \left(
      \begin{matrix}
        a & c \\
        b & d
      \end{matrix}\right) 
        = \left(
      \begin{matrix}
        a^2x+b^2z & acx+bdz \\
        acx+bdz & c^2x+d^2z 
      \end{matrix}\right) \,.
    \]
    This matrix describes the \emph{same} quadratic form as the
    diagonal matrix $\left(
      \begin{smallmatrix}
        a^2x+b^2z & 0 \\
        0 & c^2x+d^2z 
      \end{smallmatrix}\right)$.
    In other words: every diagonalizable form is in fact diagonal, and
    the action of $\GL2\KK$ on the space of diagonal forms is
    equivalent to the action on $\KK^2$ given by
    \[
    \omega^{(2)}_\KK\colon \GL2\KK\times\KK^2\to\KK\colon 
    \left(\left(
      \begin{matrix}
        a & b \\
        c & d 
      \end{matrix}\right)
    ,
    \left(
      \begin{matrix}
        x \\
        z
      \end{matrix}\right)
    \right)
    \mapsto 
    \left(
      \begin{matrix}
        a^2x+b^2z  \\
        c^2x+d^2z 
      \end{matrix}\right)
    =
    \left(
      \begin{matrix}
        a^2 & b^2 \\
        c^2 & d^2 
      \end{matrix}\right)
    \left(
      \begin{matrix}
        x \\
        z
      \end{matrix}\right)
    \,.
    \]
    The orbits under $\omega^{(2)}_\KK$ are the same as those under
    the natural action of $\GL2{\KK^\sq}$ on $\KK^2$, where $\KK^\sq$
    denotes the subfield consisting of all squares in~$\KK$.

    More generally, we consider a field extension $\LL/\KK$ such that
    $\LL^\sq\subseteq\KK$ (for instance, an inseparable quadratic
    extension $\LL/\KK$, as in~\ref{complexGHaut}.\ref{complexGHautInsepCase}
    below). In that case let $R_{\KK/\LL^\sq}^{}$ be a set of
    representatives for the cosets in $\KK/\LL^\sq$, and let
    $R^{(2)}_{\KK/\LL^\sq}$ be a set that contains precisely one
    $\LL^\sq$-basis for each $\LL^\sq$-subspace of dimension~$2$
    in~$\KK^2$.  Then the orbits under the restriction
    of~$\omega^{(2)}_\KK$ to $\GL2{\LL}\times\KK^2$ are represented by
    the elements of $\smallset{(r,0)'}{r\in R_{\KK/\LL^\sq}^{}} \cup
    R^{(2)}_{\KK/\LL^\sq}$.
 \end{ndef}

 \begin{ndef}[Hermitian Forms and Quaternions]\label{Hermite}
  Let $\LL/\KK$ be a separable quadratic extension, let
  $\sigma\colon\LL\to\LL\colon x\mapsto\widebar{x}$ denote the
  generator of the Galois group $\Gal{\LL/\KK}$.
  We will need the norm $N_{\LL/\KK}\colon\LL\to\KK\colon x\mapsto\widebar{x}x$
  and the subgroup $N_{\LL/\KK}(\LL^\times)$ of $\KK^\times$.
  
  A ($\sigma$-)hermitian form $h\colon\LL^2\to\LL$ will be described by its
  hermitian Gram matrix $M_h$ via $h \left( \left(
      \begin{smallmatrix}
        a \\ x 
      \end{smallmatrix}\right),
    \left(
      \begin{smallmatrix}
        b \\ y 
      \end{smallmatrix}\right)
     \right)
    = 
     \left( \widebar{a},\widebar{x}\right) M_h  
    \left(
      \begin{smallmatrix}
        b \\ y 
      \end{smallmatrix}\right)
  $. 
  Forms $h$ and $g$ are equivalent precisely if there exists
  $A\in\GL2\LL$ such that $\widebar{A}{}'M_hA=M_g$ where $\widebar{A}$
  is obtained from $A$ by applying $\sigma$ to each entry.
  We will call the hermitian matrices equivalent in this case. 
  Without loss, we may concentrate on the case $M_h\ne0$.
  Every hermitian matrix is equivalent to a diagonal one (necessarily,
  with entries from the ground field~$\KK$) but it is not
  easy to decide about equivalence of two given diagonal matrices, in
  general. 
  
  If $h$ is non-degenerate and isotropic then $M_h$ is equivalent to $\left(
    \begin{smallmatrix}
      0 & 1 \\
      1 & 0
    \end{smallmatrix}\right)$.
  In particular, all non-degenerate isotropic forms are equivalent to
  the one described by $\left(
    \begin{smallmatrix}
      1 & 0 \\
      0 & -1
    \end{smallmatrix}\right)$.
  Any degenerate hermitian matrix is equivalent to $\left(
    \begin{smallmatrix}
      a & 0 \\
      0 & 0
    \end{smallmatrix}\right)$
  for some $a\in\KK$, and $\left( 
    \begin{smallmatrix}
      a & 0 \\
      0 & 0
    \end{smallmatrix}\right)$ and $\left(
     \begin{smallmatrix}
      b & 0 \\
      0 & 0
    \end{smallmatrix}\right)$
  are equivalent precisely if $aN(\LL^\times)=bN(\LL^\times)$.

  It remains to understand the hermitian matrices describing
  anisotropic forms. Every such matrix is equivalent to one of the
  shape $\left(
    \begin{smallmatrix}
      a & 0 \\
      0 & c
    \end{smallmatrix}\right)$
  with determinant $ac\ne0$; and two such matrices can only be
  equivalent if their determinants are in the same coset
  modulo~$N(\LL^\times)$. However, this condition is not sufficient
  for equivalence, in general.

  For any anisotropic hermitian form $h$ on $\LL^2$ we define
  $V_h:=\smallset{h(X,X)}{X\in\LL^2\setminus\{(0,0)'\}}$. 
  The set 
  \[
  \HH_{\LL/\KK}^h 
  :=\smallset{A\in\LL^{2\times2}}{\widebar{A}{}'M_hA=M_h\det A}
  \]
  forms a quaternion field, see~\cite{MR1681303}
  (and Section~\ref{sec:quaternions} below).  Proportional forms lead to the same
  quaternion field, of course. Conversely, equality
  $\HH_{\LL/\KK}^g=\HH_{\LL/\KK}^h$ implies that there exists
  $t\in\KK^\times$ with $g=th$.
  Explicitly, we have for the form $h_c$ given by $M_{h_c}:=\left(
    \begin{smallmatrix}
      1 & 0 \\
      0 & c
    \end{smallmatrix}\right)$: 
  \[
  \HH_{\LL/\KK}^c := \HH_{\LL/\KK}^h = \set{\left(
      \begin{matrix}
        a & -c\widebar{x} \\
        x & \widebar{a}
      \end{matrix}\right)}{a,x\in\LL} \,,
  \]
  and mapping $A=\left(
    \begin{smallmatrix}
      a & -c\widebar{x} \\
      x & \widebar{a}
    \end{smallmatrix}\right)\in\HH_{\LL/\KK}^c$ to $\widetilde{A}:=\left(
    \begin{smallmatrix}
      \widebar{a} & c\widebar{x} \\
      -x & a
    \end{smallmatrix}\right)$
  gives an anti-automorphism~$\kappa$ of~$\HH_{\LL/\KK}^c$.  The map
  $N_c\colon(\HH_{\LL/\KK}^c)\to\KK\colon
  A\mapsto\det_\LL A = \widetilde{A}A$ is an anisotropic quadratic
  form (called the \emph{norm} of~$\HH_{\LL/\KK}^c$). The restriction
  of the norm to $(\HH_{\LL/\KK}^c)^\times$ is a group homomorphism
  onto the subgroup $V_{h_c}$ of~$\KK^\times$.

  We extend the definition to cover isotropic but non-degenerate
  forms, as well, considering $\HH_{\LL/\KK}^c = \set{\left(
      \begin{smallmatrix}
        a & -c\widebar{x} \\
        x & \widebar{a}
      \end{smallmatrix}\right)}{a,x\in\LL}$ also in the cases
  where there exists $(x,y)'\in\LL^2\setminus\{(0,0)'\}$ with
  $\gal{x}x+c\gal{y}y=0$. Note, however, that in this case the algebra 
  $\HH_{\LL/\KK}^c$ is smaller than the full set of matrices such that
  $\gal{A}{}'M_{h_c}A=M_{h_c}\det A$; for instance, the latter
  property is satisfied for $\left(
    \begin{smallmatrix}
      0 & x \\
      0 & y 
    \end{smallmatrix}\right)$
  whenever $\gal{x}x+c\gal{y}y=0$. The norm $N_{h_c}$ remains a
  multiplicative quadratic form but $0\in
  N_{h_c}\bigl((\HH_{\LL/\KK}^c)^\times\bigr)$ if the form $h_c$ is
  isotropic. Such a quaternion algebra will be called a %
  \emph{split quaternion algebra}.

  An element $X\in\HH_{\LL/\KK}^c$ is invertible precisely if
  $N_{h_c}(X)\ne0$; the inverse is $\frac1{N_{h_c}(X)}\widetilde{X}$. Thus
  the algebra $\HH_{\LL/\KK}^c$ is a skewfield if, and only if,
  the norm form is anisotropic. 

  In any case, we call $\tr(A):=\widetilde{A}+A$ the \emph{trace}
  of~$A\in\HH_{\LL/\KK}^c$; this is just the usual trace of
  $A\in\LL^{2\times2}$. Note that the trace describes the polar form
  $f_{N_{h_c}}$ via $f_{N_{h_c}}(X,Y) = N_{h_c}(X+Y) - N_{h_c}(X) -
  N_{h_c}(Y) = \widetilde{X}Y+\widetilde{Y}X
  = \tr\bigl(\widetilde{Y}X\bigr)$. Thus the elements perpendicular to
  the neutral element~$1$ are just those with trace~$0$.

  As usual, we identify the field $\KK$ with the subring $\KK1$
  of~$\HH_{\LL/\KK}^c$; this is just the center of the quaternion
  algebra. The embedding of $\LL$ in $\HH_{\LL/\KK}^c$ is slightly
  more delicate: we identify $a\in\LL$ with $\left(
    \begin{smallmatrix}
      a & 0 \\
      0 & \gal{a}
    \end{smallmatrix}\right)$. 
\end{ndef}

\begin{nthm}[Equivalence of Anisotropic Hermitian Forms]\label{HermiteEq}
  Let\/ $\LL/\KK$ be a separable quadratic extension, let $\sigma$ be
  the generator of the Galois group $\Gal{\LL/\KK}$, and let $g$ and
  $h$ be anisotropic $\sigma$-hermitian forms on~$\LL^2$, described by
  their hermitian Gram matrices $M_g,M_h\in\LL^{2\times2}$,
  respectively. Then the following are equivalent:  
  \begin{enum}
  \item\label{formsConjugate} The hermitian matrices (and thus the
    forms) are equivalent up to a scalar, i.e., there exists
    $A\in\GL2\LL$ with $\KK\widebar{A}{}'M_gA = \KK M_h$.
  \item\label{formsDet}
    There exists $w\in\LL^\times$ such that
    $\det M_g=N_{\LL/\KK}(w)\det M_h$.
  \item\label{quatsConjugate} The quaternion fields are conjugates,
    i.e., there is $A\in\GL2\LL$ with $A\HH_{\LL/\KK}^g
    A^{-1}=\HH_{\LL/\KK}^h$.
  \end{enum}
  The forms are equivalent if, and only if, we have $V_g=V_h$ and any
  one of the conditions~\ref{formsConjugate}, \ref{formsDet}
  or~\ref{quatsConjugate} is satisfied. 
  A form $h$ is, in particular, equivalent to $-h$ precisely if $-1\in
  N_h(\HH_{\LL/\KK}^h)$. 

  A set of representatives of equivalence classes of
  hermitian matrices is obtained as
  \[
  \left\{\left(
      \begin{matrix}
        0 & 0 \\
        0 & 0 
      \end{matrix}\right)\right\}
  \cup
  \set{\left(
      \begin{matrix}
        r & 0 \\
        0 & 0 
      \end{matrix}\right)}%
  {r\in R_{N}}
  \cup
  \left\{\left(
      \begin{matrix}
        1 & 0 \\
        0 &-1 
      \end{matrix}\right)\right\}
  \cup
  \set{\left(
      \begin{matrix}
        t & 0 \\
        0 & tr 
      \end{matrix}\right)}%
  {r\in R_{N}, t\in R_{N_r}}
  \]
  where $R_{N}\subseteq\KK^\times$ is a set of representatives for the
  cosets modulo $N_{\LL/\KK}(\LL^\times)$ and
  $R_{N_r}\subseteq\KK^\times$ is a set of representatives for the
  cosets modulo $N_{h_r}\bigl((\HH_{\LL/\KK}^{r})^\times\bigr)$ for
  the anisotropic form $h_r$ with hermitian matrix~$\left(
    \begin{smallmatrix}
      1 & 0 \\
      0 & r
    \end{smallmatrix}\right)$.
\end{nthm}
\begin{proof}
  Picking orthogonal bases with respect to $g$ or~$h$ we find
  $a,b,c,d\in\KK^\times$ such that $M_g$ is equivalent to $\left(
    \begin{smallmatrix}
      a & 0 \\
      0 & b
    \end{smallmatrix}\right)$ and $M_h$ is equivalent to $\left(
    \begin{smallmatrix}
      c & 0 \\
      0 & d
    \end{smallmatrix}\right)$.

  Condition~\ref{formsConjugate} clearly
  implies~\ref{formsDet}.
  If~\ref{formsDet} is satisfied we have
  $(ab)^{-1}cd\in N_{\LL/\KK}(\LL^\times)$. Then $\left(
    \begin{smallmatrix}
      1 & 0 \\
      0 & b/a
    \end{smallmatrix}\right)$ and $\left(
    \begin{smallmatrix}
      1 & 0 \\
      0 & d/c
    \end{smallmatrix}\right)$ are equivalent,
  and~\ref{formsConjugate} holds.

  Finally, we note that conditions~\ref{formsConjugate}
  and~\ref{quatsConjugate} are equivalent because the only $\sigma$-hermitian
  forms invariant under $\HH_{\LL/\KK}^g$ are the scalar multiples
  of~$h$.

  Every non-degenerate isotropic hermitian form on~$\LL^2$ is
  equivalent to the form $h_{-1}$ described by $M_{h_{-1}}=\left(
    \begin{smallmatrix}
      1 & 0 \\
      0 &-1
    \end{smallmatrix}\right)$. 
  The rest of the assertions is clear because $h\in\KK g$ is equivalent
  to~$g$ precisely if $V_h=V_g$. 
\end{proof}

See~\ref{exQuaternions} below for examples of quaternion fields with
different behavior with respect to the quotient
$\KK^\times/N_h\bigl((\HH_{\LL/\KK}^h)^\times\bigr)$ and also
regarding the question whether $-1$ belongs to the group
$N_h\bigl((\HH_{\LL/\KK}^h)^\times\bigr)$ of norms. 

\subsection*{Inner automorphisms of quaternion fields and similitudes
  of the norm form}

\begin{nthm}[Conjugacy in Quaternion Fields]
\label{knarrTrick}\label{conjugacyQuaternions}
  Let\/ $\HH=\HH_{\LL/\KK}^c$ be a quaternion field over a field of
  arbitrary characteristic, with norm form~$N=N_{h_c}$ and
  involution~$\kappa\colon x\mapsto\tilde{x}$.
  \begin{enum}
  \item\label{conjNorm} For $v,x\in\HH$ there exists $a\in\HH^\times$
    with $ava^{-1}=x$ precisely if $v$ and $x$ have the same norm and
    the same trace. In particular, pure quaternions (i.e., quaternions
    with vanishing trace) are conjugates if, and only if, they have
    the same norm.
  \item\label{conjIsom} Now consider $v,w,x,y\in\HH$ such that
    $w\notin\KK v$ and $y\notin\KK x$. There exists $a\in\HH^\times$
    such that $ava^{-1}=w$ and $axa^{-1}=y$ if, and only if, we have
    $N(v)=N(x)$, $N(w)=N(y)$, $\tr(v)=\tr(w)$, $\tr(x)=\tr(y)$, and
    $f_N(v,w)=f_N(x,y)$.
  \item\label{actionZ}
    There exist $a,b\in\HH^\times$ with $av\tilde{a}N(b)=x$ if, and
    only if, there exists $z\in\HH$ such that $N(x)=N(v)N(z)^2$ and
    $\tr(x)=\tr(v)N(z)$. 
  \end{enum}
\end{nthm}
\begin{proof}
  Multiplicativity of the norm form yields that conjugation with~$a$
  induces an orthogonal map on~$\HH$. It remains to prove the 
  non-trivial implications, i.e., those asserting conjugacy.
  
  Assume $N(v)=N(x)$ and $\tr(v)=\tr(x)$.
  If $x\ne\tilde{v}$ then $a:=x-\tilde{v}$ satisfies
  $\tilde{a} = \tilde{x}-v = \tr(x)-x-v = -x+\tilde{v} =
  -a$ and we compute 
  $ava^{-1} = (x-\tilde{v})va^{-1} = (xv-N(v))a^{-1} =
  (xv-N(x))a^{-1} = x(v-\tilde{x})a^{-1} = v$. If
  $x=\tilde{v}$ we pick $a\in\{1,v\}^\perp\setminus\{0\}$. Then
  $0=f_N(a,v)=\tilde{a}v+\tilde{v}a=-av+\tilde{v}a$ gives
  $ava^{-1}=\tilde{v}$.
  This proves assertion~\ref{conjNorm}.

  Under the assumptions made in~\ref{conjIsom}, we may also assume
  $x=v$ because of assertion~\ref{conjNorm}.  It suffices to consider
  the case $w\ne y$. Then $c:=y-w$ satisfies $\tilde{c} =
  (\tr(y)-y)-(\tr(w)-w) = w-y = -c$, and we compute $cwc^{-1} =
  (\tilde{w}-\tilde{y})wc^{-1} = (N(w)-\tilde{y}w)c^{-1} =
  (N(y)-\tilde{y}w)c^{-1} = \tilde{y}$. On the other hand, we
  use $\tilde{v}w+\tilde{w}v = f_{N}(v,w) =
  f_{N}(x,y) = f_{N}(v,y) = \tilde{v}y+\tilde{y}v$
  to compute $cvc^{-1} = (\tilde{w}-\tilde{y})vc^{-1} =
  (\tilde{w}v-\tilde{y}v)c^{-1} =
  (\tilde{v}y-\tilde{v}w)c^{-1} = \tilde{v}$.  It remains
  to pick $b\in\{1,v,y\}^\perp\setminus\{0\}$; then
  $b\tilde{v}b^{-1}=v$ and $b\tilde{y}b^{-1}=y$ complete the
  proof of assertion~\ref{conjIsom}.

  If $x=av\tilde{a}N(b)$ then $N(x)=N(v)N(ab)^2$ and
  $\tr(x)=\tr(ava^{-1}N(a)N(b))=\tr(v)N(ab)$. Conversely, assume that
  there is $z\in\HH^\times$ such that $N(x)=N(v)N(z)^2$ and
  $\tr(x)=\tr(v)N(z)$. Then $x$ and $vN(z)$ have the same norm and the
  same trace, and assertion~\ref{conjNorm} yields $a\in\HH^\times$
  such that $x = avN(z)a^{-1} = av\tilde{a}N(z)N(a)^{-1} = 
  av\tilde{a}N(za^{-1})$.
\end{proof}

\begin{rema}
  The assertions of~\ref{conjugacyQuaternions} do not remain valid if
  we consider a split quaternion algebra. In fact,
  such an algebra will contain divisors of zero, and the element
  $a:=x-\tilde{u}$ used in the proof
  of~\ref{conjugacyQuaternions}.\ref{conjNorm} might be non-invertible
  although $x\ne\tilde{u}$. 
  It is known that every split quaternion algebra is isomorphic to the
  algebra of $2\times2$ matrices over the ground field, with the usual
  determinant playing the role of the norm form. The example of $\left(
    \begin{smallmatrix}
      0 & -1 \\
      1 & 2 
    \end{smallmatrix}\right)$
  and $\left(
    \begin{smallmatrix}
      1 & 0 \\
      0 & 1
    \end{smallmatrix}\right)$ shows
  that indeed~\ref{conjugacyQuaternions}.\ref{conjNorm} does not
  extend to the split case without modification. 
  
  If $\Char\KK\ne2$ it remains true that pure elements (i.e., matrices
  with vanishing trace) are conjugates if, and only if, they have the
  same norm because the trace of a non-trivial central element is
  non-zero. If $\Char\KK=2$, however, our example consist of pure
  elements. 
\end{rema}

\goodbreak
The following result contains a special case of the Skolem--Noether
Theorem for central simple algebras (e.g.,
cf.~\cite[{8.4.2}]{MR770063}).

\begin{theo}\label{SkolemNoether}
  Let\/ $\HH$ be a quaternion field over a field\/ $\KK$ of arbitrary
  characteristic. Then the group $\Aut\HH$ of\/ $\KK$-linear
  automorphisms acts faithfully on the orthogonal space $\Pu\HH:=
  \set{x\in\HH}{\tr(x)=0} = 1^\perp$, inducing the group
  $\SO{}{N|_{\Pu\HH}}$. Every element of\/ $\Aut\HH$ is an inner
  automorphism.
\end{theo}
\begin{proof}
  For any $x\in\HH$ the assumption $x^2\in\KK$ implies
  $N(x)+x^2=\tr(x)x\in\KK$ and then $x\in\KK\cup\Pu\HH$. Together
  with the fact that the center of~$\HH$ is just $\KK$, the equality
  $\KK\cup\Pu\HH = \set{x\in\HH}{x^2\in\KK}$ now yields that
  $\Pu\HH$ is invariant under $\Aut\HH$.
  
  Pick $x\in\Pu\HH\setminus\KK$ and
  $y\in\Pu\HH\setminus\spnK{1,x}$. Then $\spnK{1,x}$ is a subfield
  of~$\HH$ and $xy\notin\spnK{1,x,y}$ because $xy=r+sx+ty$ with
  $r,s,t\in\KK$ implies that $(x-t)y\in\spnK{1,x}$ and then
  $y\in(x-t)^{-1}\spnK{1,x} =\spnK{1,x}$, a contradiction. Therefore,
  the set $\Pu\HH$ generates the quaternion field as a $\KK$-algebra,
  and $\Aut\HH$ acts faithfully on~$\Pu\HH$; in fact, the stabilizer
  of $x$ and $y$ is trivial.

  Since $N(x)=-x^2$ holds for $x\in\Pu\HH$ the group $\Aut\HH$ acts by
  orthogonal maps on~$\Pu\HH$. Choosing a suitable $\KK$-basis for
  $\HH$ one easily sees that the $\KK$-determinant of any left
  multiplication $x\mapsto ax$ is the square of~$N(a)$.  The right
  multiplication $x\mapsto xa=\widetilde{\tilde{a}\tilde{x}}$ then has
  determinant $N(\tilde{a})=N(a)$, and every inner automorphism has
  determinant~$1$. Therefore, $\Aut\HH$ is a subgroup of $\SO{}{N}$
  and induces a subgroup of $\SO{}{N|_{\Pu\HH}}$. 

  According
  to~\ref{conjugacyQuaternions}.\ref{conjIsom} the group of inner
  automorphisms acts transitively on any set of two-dimensional
  subspaces of given isometry type in~$\Pu\HH$. The elements $x$ and
  $y$ above were chosen outside the radical\footnote{ This precaution
    is only relevant if $\Char\KK=2$.} of the restriction of the
  polar form $f_N$ to $\Pu\HH$. Therefore the stabilizer 
  of $x$ and $y$ (as above) is trivial even in the group
  $\SO{}{N|_{\Pu\HH}}$, and the transitivity assertion
  from~\ref{conjugacyQuaternions}.\ref{conjIsom} yields that the group
  of inner automorphisms coincides with $\SO{}{N|_{\Pu\HH}}$.  

  If $\Char\KK\ne2$ it remains to show that no element of the coset\/
  $\Orth{}{N_{\Pu\HH}}\setminus\SO{}{N_{\Pu\HH}}$ is induced by an
  element of~$\Aut\HH$. As the involution $x\mapsto\tilde{x}$ is not
  an automorphism but represents the coset this is another consequence
  of the fact that each element of $\SO{}{N|_{\Pu\HH}}$ is induced by
  an inner automorphism. 
\end{proof}

\subsection*{Similitudes of quadratic forms}

\begin{defi}
  Let\/ $q\colon V\to\KK$ be a quadratic form. A \emph{similitude}
  of\/~$q$ with \emph{multiplier}~$r$ is a linear bijection
  $\lambda\colon V\to V$ such that $q(\lambda(v))=rq(v)$ holds for
  each $v\in V$. The set $\GO{}{q}$ of all similitudes forms a
  subgroup of~$\GL{}{V}$. 

  If $q\ne0$ then every multiplier is non-zero. In this case, mapping
  $\lambda\in\GO{}{q}$ to its multiplier~$\mu_\lambda$ gives a group
  homomorphism $\mu\colon\GO{}{q}\to\KK^\times$ called the
  \emph{multiplier map}.

  The range of the multiplier map contains~$\KK^\sqt$ because $s\,\id$
  is a similitude with multiplier~$s^2$.
\end{defi}

Every diagonalizable quadratic form $q\colon V\to\KK$ over a
field~$\KK$ of characteristic~$2$ is additive. Regarding~$\KK$ as a
vector space over~$\KK^\sq$ we may then view~$q$ as a semilinear map
with respect to the field isomorphism $\frob\colon\KK\to\KK^\sq\colon
x\mapsto x^2$.  The form $q$ is non-degenerate if, and only if, the
kernel of this semilinear map is trivial. Thus the form is
non-degenerate precisely if it is anisotropic.

\goodbreak
\begin{theo}\label{diagAniso}
  Let $q\colon V\to\KK$ be a diagonalizable non-degenerate quadratic
  form over a field\/~$\KK$ with $\Char\KK=2$.
  \begin{enum}
  \item\label{Otrivial}
    The form $q$ is anisotropic and $\Orth{}{q}$ is trivial.
  \item The subset\/ $\{0\}\cup\GO{}{q}$ of\/ $\End[\KK]{V}$ is a
    field, and the multiplier map $\mu$ extends to an isomorphism onto
    the
    subfield\/~$\MM:=\{0\}\cup\smallset{\mu(\lambda)}{\lambda\in\GO{}{q}}$
    of\/~$\KK$.
  \item\label{purelyInsep}
    We have $\KK^\sq\le\MM$ and\/ $\MM/\KK^\sq$ is a purely
    inseparable extension.
  \item\label{powerTwo} If\/ $\dim_\KK{V}$ is finite then\/ $\dim_\KK
    V=(\dim_\MM{V})(\dim_{\KK^\sq}\MM)$ and\/ $\dim_{\KK^\sq}\MM$ is a
    power of\/~$2$.
  \item\label{oddScalar} In particular, if\/ $\dim_\KK{V}$ is odd then
    $\GO{}{q}=\KK^\times\,\id = \mu^{-1}(\KK^\sqt)$.
  \item\label{sharpTrsGO} In any case, the group $\GO{}{q}$ acts regularly on
    $V\setminus\{0\}$. In particular, if\/ $G\le\GO{}{q}$ is
    transitive on $V\setminus\{0\}$ then $G=\GO{}{q}$.
  \end{enum}
\end{theo}
\begin{proof}
  Every isotropic vector belongs to the kernel of the semi-linear
  map~$q$. This kernel coincides with the radical of~$q$ because the
  form is diagonal. Thus $q$ is an injective map. This yields that
  $\Orth{}{q}$ is trivial, as claimed.

  For similitudes $\lambda,\sigma\in\GO{}{q}$ and $v\in V$ we compute
  $q\left((\lambda+\sigma)(v)\right) =
  q\left(\lambda(v)\right)+q\left(\sigma(v)\right) =
  \mu(\lambda)\,q(v)+\mu(\sigma)\,q(v) =
  \left(\mu(\lambda)+\mu(\sigma)\right)\,q(v)$. This shows that
  $\lambda+\sigma$ is a similitude with multiplier
  ${\mu(\lambda)+\mu(\sigma)}$ unless $\mu(\lambda)=\mu(\sigma)$. In
  the latter case we have $\lambda\sigma^{-1}\in\Orth{}{q}$ and
  $\lambda=\sigma$ by assertion~\ref{Otrivial}. Thus
  $\{0\}\cup\GO{}{q} \subseteq \End[\KK]{V}$ is additively closed and
  $\mu$ extends to an additive map onto~$\MM$. Since $\mu$ is
  multiplicative anyway, this extension is a field isomorphism.

  The similitudes in $\KK^\times\,\id$ are mapped to the elements of
  $\KK^\sqt$ under~$\mu$. For every $m\in\MM$ the minimal polynomial
  over~$\KK^\sq$ divides $X^2-m^2\in\KK^\sq[X]$, and
  assertion~\ref{purelyInsep} is proved.

  From assertion~\ref{purelyInsep} we infer that every intermediate
  field between $\KK^\sq$ and~$\MM$ either has infinite degree over
  $\KK^\sq$ or that degree is a power of~$2$. This gives the last two
  assertions~\ref{powerTwo} and~\ref{oddScalar}.

  The last assertion follows from the fact that $\Orth{}{q}$ is
  trivial. 
\end{proof}

\begin{exam}
  Let $\LL/\KK$ be a purely inseparable extension with
  $\Char\KK=2$. Then $\LL^\sq\le\KK$ and the norm
  $N_{\LL/\KK}(x)\colon \LL\to\KK\colon x\mapsto x^2$ is a quadratic
  form. This form is diagonal, and $\GO{}{N_{\LL/\KK}}=\LL^\times$
  by~\ref{diagAniso}.\ref{sharpTrsGO}. Forms of this type will occur
  in Section~\ref{sec:quaternions} below. 
\end{exam}

The following observation will be useful because it constrains the
orbits under groups of similitudes. 

\begin{lemm}\label{onlySquareFactors}
  Let\/ $q\colon V\to \KK$ be a non-degenerate quadratic form, and
  assume that $\dim V$ is odd. Then every multiplier is a square,
  and\/  $\GO{}{q}=\KK^\times\,\SO{}{q}$.
\end{lemm}
\begin{proof}
  Assume first that $\Char\KK\ne2$. Then the set of all determinants of Gram
  matrices for~$q$ is a square class which we denote by $\disc q$, and a
  similitude $\lambda$ with multiplier~$s$ will multiply $\disc q$ by
  $(\det\lambda)^2$. %
  On the other hand, we have $\disc(sq)=s^{\dim V}\disc{q}$. %
  Since $\dim V$ is odd, we immediately obtain that $s$ is a square.

  If $\Char\KK=2$ and the form is not diagonalizable then the Gram
  matrix cannot be chosen in such a way that this choice is invariant
  under linear transformations. However, the polar form $f_q$ is
  alternating in that case, and has even rank. The radical $V^\perp$
  of the polar form is invariant under any similitude, and has odd
  dimension because $\dim{V}$ is odd. The restriction of~$q$
  to~$V^\perp$ is diagonalizable, and we know from~\ref{diagAniso}
  that $s$ is a square in this case, too.

  For any $\mu\in\GO{}{q}$ with multiplier~$s^2$ we note that
  $s^{-1}\mu$ belongs to~$\Orth{}{q}$, and
  $\Orth{}{q}=\SO{}{q}\cup(-\id)\,\SO{}{q}$. Thus the last assertion
  follows.
\end{proof}

If $\dim V$ is even, the situation is more involved: depending on the
form and the field in question, non-squares may occur as multipliers.

\begin{coro}\label{fullGO3}
  For any quaternion field
  $\HH$ we have
  \begin{eqnarray*}
  \GO{}{N|_{\Pu{\HH}}} =
  \KK^\times\SO{}{N|_{\Pu{\HH}}} 
  & = &
  \set{(x\mapsto{}saxa^{-1})}{a\in\HH^\times,s\in\KK^\times} \\
  & = &
  \set{(x\mapsto{}sax\tilde{a})}{a\in\HH^\times,s\in\KK^\times}
  \,.
\qedhere
\end{eqnarray*}
\end{coro}

\section{Classification of reduced Heisenberg algebras}\label{sec:classHeis}
We need some explicit notation for the action of $\GL n\KK$ on $\Sk n$. 

\begin{nota}\label{oddNotation}
  Let $\KK$ be a field, and let $b_0,\ldots,b_{n-1}$
  denote  the standard basis for~$\KK^n$.  We will think of elements 
  $v=\sum_{j<n}v_jb_j\in \KK^n$ as columns, the transpose is written
  $v':=(v_0,\ldots,v_{n-1})$. For $v,w\in \KK^n$ we obtain the
  \emph{decomposable tensor} $v\otimes w:=vw'=(v_jw_k)_{j,k<n}$.
  The elements $b_j\otimes b_k$,
  with $j,k<n$, form the standard basis for the space $\M nn$ of
  $(n\times n)$~--~matrices with entries from~$\KK$.

  The set of \emph{alternating tensors} is the linear span $\Sk n$ of
  the elements of the form $v\wedge w:=v\otimes w-w\otimes v$.  These
  are the skew-symmetric matrices with zero diagonal (the latter
  condition follows from the former unless $\Char\KK=2$).  The
  elements $\s jk:=b_j\wedge b_k$ with $0\le j<k<n$ form a basis
  for~$\Sk n$.  With the bilinear map $\eta\colon \KK^n\times
  \KK^n\to\Sk n\colon (v,w)\mapsto v\wedge w$, we have $(\Sk n,\eta)$
  as an explicit model for the exterior product, satisfying the
  universal property: for every alternating bilinear map $\beta\colon
  V\times V\to Z$, there is a unique linear map $\hat\beta\colon \Sk
  n\to Z$ such that $\beta=\hat\beta\circ\eta$.
  
  The linear action of the group $\GL n\KK$ on $\KK^n$ induces a
  linear action on the tensor product: 
  \[
  \omega\colon \GL n\KK\times(\KK^n\otimes \KK^n)\to
  \KK^n\otimes\KK^n \colon(A,M)\mapsto AMA' \,.
  \]
  Obviously, the set $\Sk{n}$ is invariant under this action.
  We write $A.X:=AXA'$.
\end{nota}

\begin{ndef}[More Notation]
  The space $\Sk2$ is spanned by $i:=\left(
    \begin{smallmatrix}
      0 & 1 \\
      -1 & 0 
    \end{smallmatrix}\right) = (1,0)'\wedge(0,1)'$.
  We will use this notation for blocks in elements of $\Sk4$ later
  on. By $E_j$ we denote the identity matrix in $\KK^{j\times j}$. 
\end{ndef}

\begin{ndef}[The Pfaffian Form and the Klein Quadric]\label{O3Kdef}
  We will be interested in the Grassmann space $\G d4$ of
  $d$-dimensional subspaces of ${\Sk4}$, where $d\in\{1,2,3,4,5\}$.
  Note that $\omega$ induces an action on each one of the Grassmann
  spaces. 
  The following belongs to classical line geometry; details and proofs
  may also be found in~\cite{MR2431124}.
  
The group $\GL4\KK$ acts with exactly three orbits on $\Sk 4$,
represented by $0$, $\s01$ and $\s01+\s23$. Therefore, there are two orbits
on~$\G14$. The orbit of $\smallspnK{\s01}$ consists of subspaces $\smallspnK{X}$ with 
$X\in\Sk4\setminus\{0\}$ and $\det X=0$. We use the basis
$\s01,\s02,\s03,\s12,\s13,\s23$ to introduce homogeneous
coordinates  
$[x_0,\ldots,x_5]$ for~$\smallspnK{X}$, where 
\[
 X =    x_0\s01 + x_1\s02 + x_2\s03 
                + x_3\s12 + x_4\s13 
                          + x_5\s23 
   = 
 \left(\begin{array}{cccc}
          0 &  x_0 &  x_1  &  x_2 \\
       -x_0 &   0  &  x_3  &  x_4 \\ 
       -x_1 & -x_3 &   0   &  x_5 \\
       -x_2 & -x_4 & -x_5  &   0  
         \end{array}\right) \,.
\]
The orbit of $\smallspnK{\s01}$ may then be described as the quadric
$Q$ (known as the Klein quadric) defined by the \emph{Pfaffian form}
$\pfaff(x_0,x_1,x_2,x_3,x_4,x_5):=x_0x_5-x_1x_4+x_2x_3$,
cf.~\cite[\S\,5, no.\,2]{MR0107661}. This is a quadratic form of Witt index~$3$ on
$\Sk4$.  The complement of~$Q$ is the second orbit in~$\G14$.

We re-arrange the basis, using $\s01,\s02,\s03,\s23,-\s13,\s12$. With
respect to the new basis, the Pfaffian form itself may be described 
as $\pfaff(v)=v'M_\pfaff v$, 
and the polar form~$f_\pfaff$ becomes $f_\pfaff(v,w)=v'Jw$ with 
the Gram matrix~$J$, where~$M_\pfaff,J\in\KK^{6\times6}$ are defined as
$M_\pfaff:=\left(
  \begin{smallmatrix}
    0 & E_3 \\
    0 & 0 
  \end{smallmatrix}\right)$
and $J:=M_\pfaff+M_\pfaff':=\left(
  \begin{smallmatrix}
    0 & E_3 \\
    -E_3 & 0 
  \end{smallmatrix}\right)$. 

The group $\GL4\KK$ acts by similitudes with respect to~$\pfaff$.  This
yields a homomorphism $\delta$ from $\GL4\KK$ to $\GO{}{\pfaff}$ with
kernel $\{\id,-\id\}$.  We will use the induced groups $\PGL4\KK$ and
$\PGO6\KK$ on the projective spaces (or on the quadric): the
homomorphism $\delta$ induces an isomorphism from $\PGL4\KK$ onto a
subgroup $\PGOplus{}{\pfaff}$ of index~$2$ in~$\PGO{}{\pfaff}$,
see~\cite[3.11]{MR2431124}.  
\end{ndef}

The Klein quadric provides a model for the space $\mathcal L$ of lines in the
$3$-dimensional projective space over~$\KK$ via the bijection
$\lambda\colon \mathcal L\to Q\colon \smallspnK{u,\,v}\mapsto\smallspnK{u\wedge v}$.
This can be used to understand the action of $\GL4\KK$ as
$\PGOplus{}{\pfaff}$, cf.~\cite[3.4\,ff]{MR2431124}: 

\begin{lemm}\label{genericOrbits} 
  \begin{enum}
  \item\label{pairsOfLines} The group\/ $\GL4\KK$ acts with precisely
    three orbits on the set of pairs of lines, represented by the pairs\/
    $(L_0,L_0)$, $(L_0,K_0)$, and\/ $(L_0,K_1)$, where
    $L_0:=\smallspnK{b_0,b_1}$ and\/ $K_j:=\smallspnK{b_0+jb_2,b_3}$.
  \item\label{confluence}
    Two lines $K,L\in\mathcal L$ share a point if, and only if, their
    images $\lambda(K)$ and $\lambda(L)$ are orthogonal with respect
    to~$q$.
  \item 
  The maximal totally singular subspaces with respect to~$q$ are just the
  images of maximal sets of pairwise confluent lines. There are two types of
  such sets: pencils $\mathcal L_p:=\set{L\in\mathcal L}{p<L}$ and, dually,
  line sets of planes $\mathcal L_P:=\set{L\in\mathcal L}{L<P}$. 
  \item\label{Mthree}
  The action of~$\GL4\KK$ on the set $\mathcal M_3$ of maximal totally
  singular subspaces has two orbits, represented by
  $\lambda(\mathcal L_p) =\smallspnK{\s01,\,\s02,\,\s03}$ and
  $J(\lambda(\mathcal L_p))=\lambda(\mathcal L_P)=\smallspnK{\s12,\,\s13,\,\s23}$,
  where $p=\smallspnK{b_0}$, and $P=\smallspnK{b_1,\,b_2,\,b_3}$. 
  \item\label{Mtwo} The group $\GL4\KK$ acts transitively on the set
    $\mathcal M_2$ of $2$-dimensional maximal totally singular
    subspaces.
    We may use $\lambda(\mathcal L_p) \cap \lambda(\mathcal L_{P'}) =
    \smallspnK{\s01,\,\s02}$ as a representative. 
\qed
  \end{enum}
\end{lemm}

The orbits on $\G24$, $\G34$, $\G44$ and $\G54$ (i.e., on the sets of lines,
planes, three-spaces, and hyperplanes, respectively, in the projective space
$\mathcal P$ coordinatized by~$\Sk4$) may be described using the Klein
quadric~$Q$. We introduce some more notation.   

\begin{defs}\label{ETSdef}
  We consider the following lines in~$\mathcal P$: 
\[
\begin{array}{rclcrclcrcl}
 E &:=& \smallspnK{\s01,\, \s02}, && T &:=& \smallspnK{\s01,\, \s03+\s12}, &\mbox{and}&
 S &:=& \smallspnK{\s01,\, \s23} \,.
\end{array}
\]
  The orthogonal spaces (with respect to~$q$) are 
\[
\begin{array}{rcll}
   E^\perp&=&\smallspnK{\s01,\, \s02,\, \s03,\, \s12} \,, \\[.5em]
   T^\perp&=&\smallspnK{\s01,\, \s02,\, \s03-\s12,\, \s13} \\[.5em]
   \text{and}\quad
   S^\perp&=&\smallspnK{\s02,\, \s03,\, \s12,\, \s13} \,, & \text{respectively.}
\end{array}  
\]
  We will also use the planes
\[
\begin{array}{rclcrcl}
    F &:=& \smallspnK{\s01,\, \s02,\, \s03} \,,      &&
  E+T &:=& \smallspnK{\s01,\, \s02,\, \s03+\s12} \,, \\[.5em]
  E+S &:=& \smallspnK{\s01,\, \s02,\, \s23} \,,      &&
  T+S &:=& \smallspnK{\s01,\, \s03+\s12,\, \s23} \,.
\end{array}
\]
With respect to the given bases, the restriction of~$\pfaff$ to
the subspace~$X$ may be described by an upper triangular matrix $m_X$, where 
\[
\begin{array}{cccc}
 m_E=\left(\begin{array}{cc}0&0\\0&0\end{array}\right) , &
 m_T=\left(\begin{array}{cc}0&0\\0&1\end{array}\right) , &
 m_S=\left(\begin{array}{cc}0&1\\0&0\end{array}\right) , &
 m_{F}  =\left(\begin{array}{ccc}0&0&0\\0&0&0\\0&0&0\end{array}\right) ,
\end{array}
\]
\[
\begin{array}{cccc}
 m_{E+T}=\left(\begin{array}{ccc}0&0&0\\0&0&0\\0&0&1\end{array}\right) ,&
 m_{E+S}=\left(\begin{array}{ccc}0&0&1\\0&0&0\\0&0&0\end{array}\right) ,&
 m_{T+S}=\left(\begin{array}{ccc}0&0&1\\0&1&0\\0&0&0\end{array}\right) , 
\end{array}
\]
\[
\begin{array}{ccc}
 m_{E^\perp}=\left(
   \begin{array}{cccc}
     0&0&0&0\\
     0&0&0&0\\
     0&0&0&1\\
     0&0&0&0\\
   \end{array}\right) \,,&
 m_{T^\perp}=\left(
   \begin{array}{cccc}
     0&0&0&0\\
     0&0&0&-1\\
     0&0&1&0\\
     0&0&0&0\\
   \end{array}\right) \,,&
 m_{S^\perp}=\left(
   \begin{array}{cccc}
     0&0&0&-1\\
     0&0&1&0\\
     0&0&0&0\\
     0&0&0&0\\
   \end{array}\right) \,.
\end{array}
\]
Note that the Gram matrix for the polar form of the restriction is
$m_X+(m_X)'$. 

In order to describe certain subspaces that have only a small
intersection with the Klein quadric, we use 
\begin{eqnarray*}
  \formspace act
  & := &
  \spnK{\s01+a\s23,\,\s03+c\s12+t\s23 }
  \\
  \Formspace abct
  & := &
    \spnK{\s01+a\s23,\,
    \s02-b\s13,\,
    \s03+c\s12+t\s23 } \,.
\end{eqnarray*}
\end{defs}

The following has been proved in~\cite[5.6]{MR2431124}.

\begin{theo}\label{orbitsFour}
\begin{enum}
\item The $\GL4\KK$-orbits in $\G14$ are represented by $\smallspnK{\s01}$
  and~$\smallspnK{\s01+\s23}$. 
\item The $\GL4\KK$-orbits in $\G24$ are represented by a set 
  $\{E,T,S\}\cup\mathcal P_1$, where $\mathcal P_1$ denotes a
  (possibly empty) set of nonsingular lines. 
\item The $\GL4\KK$-orbits in $\G34$ are represented by a set 
  $\{F,J(F),E+T,E+S,T+S\}\cup\mathcal P_2\cup\mathcal P_3$, where $\mathcal P_2$
  denotes a (possibly empty) set of nonsingular planes, and $\mathcal P_3$
  is a (possibly empty) set of planes of the form $\smallspnK{\s02}+\ell$, where
  $\ell$ is a nonsingular line contained in $(\s02)^\perp$.  
\item The $\GL4\KK$-orbits in $\G44$ are represented by 
  $\mathcal U^\perp:=\smallset{U^\perp}{U\in\mathcal U}$, where
  $\mathcal U$ is an arbitrary set of representatives in~$\G24$. 
\item The $\GL4\KK$-orbits in $\G54$ are represented by ${\smallspnKperp{\s01}}$
  and~${\smallspnKperp{\s01+\s23}}$. 
\end{enum}\noindent
For $\ker\hat\beta\in\{F,E^\perp,{\smallspnKperp{\s01}}\}$ the
Heisenberg algebra $\gheis{\KK^4}{({\Sk4})/\ker{\hat\beta}}{\beta}$ is
not reduced.  
If\/ $\ker\hat\beta$ runs over the remaining representatives we obtain
a complete system of pairwise non-isomorphic
reduced Heisenberg algebras of dimension~$4$ modulo their center. 
\qed
\end{theo}

\begin{rems}\label{orbitsFourSpecial}
  One may always choose 
  $\mathcal P_1\cup\mathcal P_2\cup\mathcal P_3$ using members from
  the collection of spaces $\formspace act$ or $\Formspace abct$ (in
  fact, we could choose \emph{all} our representatives to be of this
  form, but prefer $E$, $S$ and~$T$). We
  indicate some special cases, cf.~\cite[5.3]{MR2431124}: 
\begin{enum}
\item   
  If $\KK$ is a euclidean field (e.g. $\KK=\RR$)  we
  may choose $\mathcal P_1:=\{\formspace 110\}$,
  $\mathcal P_2:=\{\Formspace 1110\}$, and
  $\mathcal P_3:=\{\Formspace 1010\}$.
\item\label{passingLinesAndFieldExtensions}
  Let $d,t$ be chosen such that the restriction of~$\pfaff$
  to~$\formspace 1dt$ is anisotropic. Then $X^2+tX+d$
  is an irreducible polynomial over~$\KK$.
  The Heisenberg algebras corresponding to the spaces
  $\bigl(\formspace 1d0\bigr)^\perp =
  \smallspnK{\s02,\s13,\s01-\s23,\s03-d\s12}$ and
  $\bigl(\formspace 1d1\bigr)^\perp =
  \smallspnK{\s02,\s13,\s01-\s23-\s12,\s03-d\s12}$ may be interpreted as
  Heisenberg algebras of dimension~$3$ over the field extension~$\LL$
  obtained by adjoining a root of $X^2+tX+d$ to~$\KK$,
  cf.~\cite[{8.3}]{MR2431124}. In~\ref{complexGH} below we give an
  explicit construction, starting from an embedding of $\LL$ into the
  ring $\KK^{2\times2}$ and such that the corresponding embedding of
  $\GL2\LL$ by block matrices in $\GL4\KK$ leaves a suitable element
  of the orbit of $\formspace1dt$ invariant.
  This explicit description also helps to understand the elements
  of~$\mathcal P_3$, see~\ref{understandP3} below.  
\item\label{passingPlanesAndQuaternions}
  There is a connection between nonsingular subspaces of dimension~$3$
  and Heisenberg algebras defined using a quaternion field over~$\KK$, 
  see~\cite[{8.4}]{MR2431124} and Section~\ref{sec:quaternions} below.
  If $\Char\KK\ne2$ then every anisotropic subspace of dimension~$3$ in $\Sk4$
  belongs to the orbit of $\Formspace 1uv0$ for suitable $u,v\in\KK$
  (see~\cite[8.6]{MR2431124}), 
  and thus to the orbit of $\ker{\widehat{\beta_\HH}}$ for the
  quaternion field $\HH:=\HH_\KK^{u,v}$. 
  This connection allows to describe the
  relevant part of the automorphism groups of these Heisenberg algebras,
  see~\ref{SigmaBetaHH} and~\ref{omegaQuaternionField} below. 
\item
  If $\KK$ is quadratically closed (e.g., if $\KK=\CC$) then there
  are no nonsingular lines, and each of the sets~$\mathcal P_j$ is empty.
\item
  If $\KK$ is finite we may chose $\mathcal P_1=\{\formspace 1dt\}$,
  $\mathcal P_2:=\emptyset$, and~$\mathcal P_3:=\{\Formspace 10dt\}$,
  for any irreducible polynomial $X^2+tX+d$ over~$\KK$.   
\end{enum}
\end{rems}

\section{The cases where $\Sigma_\beta$ can be computed directly}
\label{sec:directComputations}

In order to determine $\Aut{\gheis{V}{Z}{\beta}}$ for a reduced
Heisenberg algebra $\gheis{V}{Z}{\beta}$ it suffices to determine
$\Sigma_\beta$. This is comparatively easy if $V=\KK^4$ and the kernel
of $\hat\beta$ is one of the ``generic'' subspaces\footnote{We only
  consider subspaces that belong to \emph{reduced} Heisenberg
  algebras, and ignore $F$, $E^\perp$, $\smallspnKperp{\s01}$,
  cf.~\ref{orbitsFour}.} of $\Sk4$, i.e., belongs to the orbit
of~$\smallspnK{\s01}$, $\smallspnK{\s01+\s23}$, $E$, $T$, $S$, $J(F)$, $E+T$,
$E+S$, $T+S$, or the orthogonal space of one of $\smallspnK{\s01}$,
$\smallspnK{\s01+\s23}$, $E$, $T$, or~$S$.  In fact, each of these spaces
(except $\smallspnK{\s01+\s23}$) has considerable intersection with the
Klein quadric, and this intersection is invariant under the stabilizer
$\Sigma_\beta$. The resulting conditions mean that
$\Sigma_\beta\le\GL4\KK$ fixes a certain system of lines of the
$3$-dimensional projective space in each of these cases; and this
allows to determine~$\Sigma_\beta$. Details are given in this section.

The situation is more complicated if $\hat\beta$ meets the Klein
quadric in only a few points, or none at all. This situation requires
the existence of anisotropic forms in $2$ or $3$ variables; thus it
only occurs for specially chosen fields, and is not ``generic''.  We
postpone the discussion of these cases to
Sections~\ref{sec:fieldextensions} and~\ref{sec:quaternions} where we
will use field extensions and quaternion algebras associated to these
anisotropic forms to determine~$\Sigma_\beta$.

For any subspace $X\le\Sk4$ the orthogonal space $X^\perp$ has the same
stabilizer because $\GL4\KK$ acts by similitudes.
In most of the following cases, the transitivity assertions that are
implicit in the statements about the orbit representatives are easy to
check by a direct computation: one has to apply the respective stabilizer
to the given representatives. 

\begin{prop}\label{stabS01}
  The stabilizer of $\smallspnK{\s01}$ is the stabilizer of the line
  $\smallspnK{b_0,b_1} = \lambda^{-1}(\s01)$: 
  \[
  (\GL4\KK)_{\smallspnK{\s01}} = \set{\left(
      \begin{array}{cc}
        A & B  \\
        0 & C  \\
      \end{array}\right)}{%
      A,C\in\GL2\KK,\,  B\in\KK^{2\times2}  }
  = (\GL4\KK)_{\smallspnKperp{\s01}} \,. 
  \]
The orbits of this stabilizer on~$\KK^4$ are represented by~$0$,
$b_0$, $b_2$, those on $(\Sk4)/\smallspnK{\s01}$ by $\smallspnK{\s01}$,
$\s02+\smallspnK{\s01}$, $\s02+\s13+\smallspnK{\s01}$, $\s23+\smallspnK{\s01}$ and those
on $(\Sk4)/\smallspnKperp{\s01}$ by $\smallspnKperp{\s01}$,
$\s23+\smallspnKperp{\s01}$. 
\qed
\end{prop}

The space $\smallspnK{\s01+\s23}$ will be considered next. This describes a
point outside the Klein quadric~$Q$ but the orthogonal space
$\smallspnKperp{\s01+\s23}$ has large intersection with~$Q$. The treatment
will be complicated; this is due to the fact that we discuss the
representation of $\Sp4\KK$ as a group of orthogonal transformations
on a space of dimension~$5$ (which is, indeed, the space
$\smallspnKperp{\s01+\s23}$). This representation gives rise to one of the interesting
exceptional isomorphisms between classical groups, corresponding to the
isomorphism of simple Lie algebras of types $\mathsf{C}_2$ and
$\mathsf{B}_2$. 

\begin{prop}\label{stabGSp}
  The stabilizer of $\smallspnK{\s01+\s23}$ is a conjugate of that 
  of $\N:=\smallspnK{\s02+\s13}$, and
  \[
  (\GL4\KK)_{\N} = \set{\left(
      \begin{array}{cc}
        A & B  \\
        C & D  \\
      \end{array}\right)}{\begin{array}{cc}
      A,B,C,D\in\KK^{2\times2}, \\
      AB'-BA' = 0 = CD'-DC' \\
      \exists s\in\KK^\times\colon AD'-BC' = sE_2
    \end{array} }
  = (\GL4\KK)_{\Nperp} \,. 
  \]
  This is the group $\GSp4\KK$ of all similitudes of the
  non-degenerate alternating form mapping $(x,y)\in\KK^4\times\KK^4$
  to~$x'(\s02+\s13)y$. %
  The group $\GSp4\KK$ acts transitively both on $\KK^4\setminus\{0\}$
  and on $\left((\Sk4)/\Nperp\right)\setminus\{0\}$.  In order to
  describe the orbits on $\N$, on the orthogonal space $\Nperp =
  \smallspnK{\s01,\s02-\s13,\s03,\s23,\s12}$, and on $(\Sk4)/\N$ we choose a
  set $R_*$ of representatives for the cosets forming the
  multiplicative group $\KK^\times/\KK^\sqt$ of
  square classes; here 
  $\KK^\sqt:=\smallset{s^2}{s\in\KK^\times}$ is the
  multiplicative group of squares. If $\Char\KK=2$ we need to pick a set $R_+$ of
  representatives for the orbits of $\KK^\sqt$ on the additive group
  $\KK/\KK^\sq$ where $\KK^\sq:=\smallset{x^2}{x\in\KK}$ is
  the subfield of squares, and a set $R_\wp$ of representatives
  for the cosets in the additive group $\KK/\wp$, cf.~\ref{Arf}. 

  We have to distinguish cases:
  \begin{enum}
  \item\label{onlySquareClass}
    If an orbit of $\GSp4\KK$ on $\Sk4$ contains an element $X$ with
    $q(X)\ne0$ then $q$ assumes on that orbit only values
    from a single square class. 
  \item\label{charOddSpOrbits}
    If\/ $\Char\KK\ne2$ then the orbits on $(\Sk4)/\N$
    are represented by the elements of the set\/
    $\mathcal{R}_* := \{\N\} \cup
    \smallset{\s01+r\s23+\N}{r\in R_*\cup\{0\}}$.
  \item\label{charEvenSpOrbits} If\/ $\Char\KK=2$
    then the orbits in $\Nperp/\N$ are represented by the elements
    of\/ $\mathcal{R}_{+}:=\{N\}\cup\set{(\s01+r\s23)+N}{r\in R_+}$,
    and those in
    $\left((\Sk4)/N\right)\setminus\left(\Nperp/\N\right)$ are
    represented by
    $\mathcal{R}_{\wp}:= 
     \{\s13+N\}\cup\set{\s13+(\s01+r\s23)+N}{r\in R_\wp}$. 
  \item\label{formOnlySquareClass} In any case, the orbits on $\N$ are
    represented by $\set{r(\s02+\s13)}{r\in R_*\cup\{0\}}$.
  \end{enum}  

    Note that $\mathcal{R}_{+}$ is an infinite set whenever $\KK$ is
    not a perfect field: the additive group of~$\KK$ forms a vector
    space over the subfield\/ $\KK^{\sq}$ of squares; in the corresponding affine
    space the set $\mathcal{R}_{+}$ represents
    the line $\KK^\sq 1$ and the $\KK^\sq$-planes through that
    line.
    
    The set\/ $R_\wp$ can be chosen inside $\KK^\sq$ because $r^2$
    belongs to $r+\wp$.
\end{prop}

\begin{proof}
  Only the assertions about the orbits need a proof.
  For the sake of easy reference later on, we note that the stabilizer
  of $\N$ contains the subgroups 
  \[
  \Delta  := 
  \set{\left(
      \begin{array}{cc}
       aA & 0 \\
        0 & (A^{-1})'
      \end{array}\right)}%
  { \begin{array}{ll}
      A\in\GL2\KK\\ a\in\KK^\times 
    \end{array}} , \,
  \Lambda  :=  
  \set{\left(
      \begin{array}{cc}
        E_2 & 0 \\
        X & E_2 
      \end{array}\right)}
  { \begin{array}{ll}
      X\in\KK^{2\times2}\\ X'=X
    \end{array}} ,
  \,
  \Upsilon  :=  \Lambda'
  \,.
  \]
  It is well known\footnote{From Witt's Theorem,
    see~\cite{MR770063}~{\S\,9} or~\cite{MR0072144}~{\S\,11, p.\,21
      and \S\,16,~p.\,35}, cf.~\cite{MR1189139}~{7.4} or
    \cite{MR0082463}~{Thm.~3.9} for $\Char\KK\ne2$.} 
  that $\Sp4\KK$ acts transitively
  on~$\KK^4\setminus\{0\}$. We offer a direct argument: let
  $x,y\in\KK^2$ and assume that not both of these vectors are~$0$. Our
  claim then is that there exists
  $\left(\begin{smallmatrix}
      A & B \\
      C & D
    \end{smallmatrix}\right)\in\Sp4\KK\subseteq\GSp4\KK$ mapping
  $\left(
    \begin{smallmatrix}
      x \\ y
    \end{smallmatrix}\right)$ to $(1,0,0,0)'$.
  If $y=0$ we choose $A\in\GL2\KK$ with $Ax = \left(
    \begin{smallmatrix}
      1 \\ 0 
    \end{smallmatrix}\right)$, put $D:=(A')^{-1}$ and $B=0=C$.
  If $y\ne0$ we find $B\in\GL2\KK$ such that
  $By=\left(
    \begin{smallmatrix}
      1 \\ 0 
    \end{smallmatrix}\right)$, and a symmetric matrix
  $S\in\KK^{2\times2}$ with $S\left(
    \begin{smallmatrix}
      1 \\ 0 
    \end{smallmatrix}\right)=-(B')^{-1}x$.
  Now $A:=0$, $C:=-(B')^{-1}$ and $D:=-SB$ yields an element of
  $\Sp4\KK$ with the required properties. 

  The cosets in $(\Sk4)/\Nperp$ are represented by
  elements of $\smallspnK{\s02}$; the nontrivial representatives are in a
  single orbit under the subgroup $\set{\left(
      \begin{smallmatrix}
        sE_2 & 0 \\
        0 & E_2
      \end{smallmatrix}\right)}{s\in\KK^\times}$ of $\GSp4\KK$.  The same
  subgroup also shows that $\set{r(\s02+\s13)}{r\in R_*\cup\{0\}}$
  contains a set of representatives for the orbits
  on~$\N$. Different elements of this set can not be
  fused into the same orbit because the values $q(r(\s02+\s13))=-r^2$
  all belong to the same square class; and
  assertion~\ref{formOnlySquareClass} is proved. We note that this observation
  also implies assertion~\ref{onlySquareClass}: On each
  $\GSp4\KK$-orbit on $\Sk4$ the form $q$ assumes only values from a
  single square class, or only the value~$0$.

  It remains to determine the orbits on $(\Sk4)/\N$. 
  We consider an arbitrary element
  $M=\left(
    \begin{smallmatrix}
      ai & B \\
      -B' & ci
    \end{smallmatrix}\right)$ of $\Sk4$.
  Note that $q(M)=ac-\det B$. 
  Assume first $M\in\Nperp$; then $B\in\KK^{2\times2}$
  has trace~$0$ and $-B'i$ is symmetric.
  We search for a representative with $B=0$ in the
  orbit of~$M$, so assume $B\ne0$.

  If $a\ne0$ we use  $\left(
    \begin{smallmatrix}
      E_2 & 0  \\
       -a^{-1}B'i\, & E_2
     \end{smallmatrix}\right) \in\Lambda$ %
   to transform $M$ into $\left(
     \begin{smallmatrix}
       ai & 0 \\
       0 & di 
     \end{smallmatrix}\right)$, where $d=c-a^{-1}\det B$. %
   A suitable element of~$\Delta$ maps this into an element of
   $\set{\s01+r\s23}{r\in R_*\cup\{0\}}$. %
   The case $a=0\ne c$ is reduced to the previous one by an
   application of $\left(
    \begin{smallmatrix}
      0 & E_2  \\
      -E_2 & 0
    \end{smallmatrix}\right) \in\GSp4\KK$. %

  If $a=0=c$ we may assume that $B$ is not a scalar multiple of~$E_2$
  because otherwise $M$ belongs to $0+\N\in\mathcal{R}_*$.
  We pick a symmetric matrix~$X$ such that $BX$ is not symmetric.
  Then $\left(
    \begin{smallmatrix}
      E_2 & X \\
       0  & E_2
     \end{smallmatrix}\right)\in\Upsilon$ %
   can be used to transform $M$ into a matrix with $a\ne0$.

   Thus we have shown that $\mathcal{R}_*$ contains a full set of
   representatives for the orbits of cosets of elements in
   $(\Nperp+\N)/\N$.

   If $\Char\KK\ne2$ then the orthogonal space is a complement
   to~$\N$ and of course invariant under the stabilizer
   of~$\N$. 
   Different elements of $\mathcal{R}_*$ belong
   to different orbits because their images under~$q$ belong to
   different square classes (cf. assertion~\ref{onlySquareClass}),
   and assertion~\ref{charOddSpOrbits} of the
   theorem is established.
   
   The situation changes in two respects if $\Char\KK=2$: since
   $\N\le\Nperp$ it is possible that different elements of
   $\mathcal{R}_*$ belong to the same orbit, and we have to consider
   representatives outside $\Nperp$ as well.

   We apply $\left(
     \begin{smallmatrix}
       E_2 & 0 \\
       a^2i\, & E_2 
     \end{smallmatrix}\right)\in\Lambda$ to $(\s01+r\s23)+\N$ and
   obtain $(\s01+(r+a^2)\s23)+a(\s02+\s13)+\N =
   (\s01+(r+a^2)\s23)+\N$. This shows that the orbits in $\Nperp/N$
   are represented by the elements of
   $\{\N\}\cup\set{(\s01+r\s23)+\N}{r\in R_+}$, as claimed. 

   In order to treat orbits without representatives in $\Nperp$ we
   consider $M=\left(
    \begin{smallmatrix}
      ai & T \\
      -T' & ci
    \end{smallmatrix}\right)$
  where the trace~$t$ of~$T$ is different from~$0$.
  We write $p:=\left(
     \begin{smallmatrix}
       0 & 0 \\
       0 & 1
     \end{smallmatrix}\right)$
   and find that $B:=T-tp$ has trace~$0$; thus $Bi$ is symmetric.
   
   If $a\ne0$ we use $\left(
     \begin{smallmatrix}
       E_2 & 0 \\
       -a^{-1}B'i\, & E_2 
     \end{smallmatrix}\right) \in\Lambda$
   to find $\left(
     \begin{smallmatrix}
       ai & tp \\
       -tp & di 
     \end{smallmatrix}\right)$ in the orbit of~$M$.
   Now we use $A:=\left(
     \begin{smallmatrix}
       t^{-1} & 0 \\
       0 & t^{-1}a 
     \end{smallmatrix}\right) \in \GL2\KK$ and
   $\left(
     \begin{smallmatrix}
       tA & 0 \\
       0 & A 
     \end{smallmatrix}\right) \in \Delta$
   to transform  $\left(
     \begin{smallmatrix}
       ai & tp \\
       -tp & di 
     \end{smallmatrix}\right)$
   to  $\left(
     \begin{smallmatrix}
       i & p \\
       -p & si 
     \end{smallmatrix}\right)$
   with $s=t^{-2}ad$.  As above, the cases with $a=0$ are reduced to
   this case; note that $T\notin\KK E_2$.

   It remains to identify sets of representatives from
   $\mathcal{T}_{+}:=\set{(\s01+r\s23)+N}{r\in\KK}$ and
   $\mathcal{T}_\wp:=\{\s13+N\}\cup\set{\s13+(\s01+r\s23)+N}{r\in\KK}$.
   Since $\Nperp$ is invariant under the action, we know that an
   element of $\mathcal{T}_+$ is never in the orbit of an element
   of~$\mathcal{T}_\wp$. 

   We compute $q\left(a\s13+(\s01+r\s23)+N\right)=\set{r+u^2+au}{u\in\KK}$.
   Since $\GSp4\KK$ can change the values of the Pfaffian form only by
   square factors we know that $a\s13+(\s01+r\s23)+N$ and $a\s13+(\s01+t\s23)+N$
   can only belong to the same orbit if there exists $v\in\KK^\times$ such that
   $v^2t=r+u^2+au$.

   If $a=0$ we apply the element $\left(
     \begin{smallmatrix}
       E_2 & 0 \\
       0 & vE_2
     \end{smallmatrix}\right)\in\Delta$ to see that indeed the
   elements $(\s01+r\s23)+N$ and $(\s01+t\s23)+N$ belong to the same
   orbit if $s\in \KK^\sqt(r+\KK^\sq)$. Thus we have proved that
   $\mathcal{R}_+$ represents the orbits inside $\mathcal{T}_+$.

   If $a=1$ we use $F:=\left(
    \begin{smallmatrix}
      E_2 & 0 \\
      ui & E_2
    \end{smallmatrix}\right)$
   which belongs to~$\Lambda$ because $\Char\KK=2$, and observe
   $F\left(\s13+(\s01+r\s23)\right)F' =
   \s13+\s01+(r-u^2)\s23+u(\s02+\s13) = \s13+\s01+t\s23+u(\s02+\s13)$
   as required.

   Finally, assume that $\s13+(\s01+r\s23)+N$ and
   $\s13+(\s01+t\s23)+N$ belong to the same orbit under the stabilizer
   of~$N$. Then there exists a multiplier $u^2\in\KK^\sqt$
   such that $r+\wp = q\left(\s13+(\s01+r\s23)+N\right) =
   u^2\,q\left(\s13+(\s01+t\s23)+N\right) = u^2t+u^2\wp$. This yields
   $\wp=u^2\wp$. We claim that $u=1$ if $r+\wp\ne t+\wp$.  In fact,
   for each $x\in\KK$ we have $u^2(x^2+x)\in u^2\wp=\wp$ and then
   $(u+1)ux = u^2(x^2+x)+\left((ux)^2+(ux)\right) \in\wp+\wp=\wp$.
   Thus $(u+1)u\,\KK\subseteq\wp$, and $(u+1)u=0$ follows because
   $r+\wp\ne t+\wp$ implies $\wp\ne\KK$. Now $u=1$ is the only
   solution for $(u+1)u=0$ in $\KK^\times$. 
\end{proof}

\begin{prop}\label{stabE}
  The stabilizer $(\GL4\KK)_E$ of\/ $E$ coincides with the stabilizer of
  the point-plane flag $(\smallspnK{b_0},\smallspnK{b_0,b_1,b_2})$. Thus
  \[
  (\GL4\KK)_E = \set{\left(
      \begin{array}{ccc}
        a & b'& c \\
        0 & D & e \\
        0 & 0 & f
      \end{array}\right)}{\begin{array}{cc}
      a,c,f\in\KK, & af\ne0, \\ b,e\in\KK^2, & D\in\GL2\KK
    \end{array}}
  = (\GL4\KK)_{E^\perp} \,. 
  \]
  The orbits on~$\KK^4$ are represented by~$0$, $b_0$,
  $b_1$, $b_3$, those on $(\Sk4)/E$ are represented by $E$,
  $\s03+E$, $\s12+E$, $\s13+E$, and those on $(\Sk4)/E^\perp$ by
  $E^\perp$, $\s13+E^\perp$. 
\end{prop}
\begin{proof}
  The statement about the stabilizer follows from the fact that $E$ corresponds
  (via $\lambda$, see~\ref{genericOrbits}.\ref{Mtwo}) to the set of lines passing
  through~$\smallspnK{b_0}$ and lying in the plane~$\smallspnK{b_0,b_1,b_2}$.
\end{proof}

\begin{prop}\label{stabT}
  The stabilizer $(\GL4\KK)_T$ of\/ $T$ is
  \[
  (\GL4\KK)_T = \set{\left(
      \begin{array}{cc}
        A & B \\
        0 & c\sigma A\sigma  
      \end{array}\right)}{\begin{array}{cc}
      A\in\GL2\KK, &  B\in\KK^{2\times2}, \\
      c\in\KK\setminus\{0\} \end{array}}
  = (\GL4\KK)_{T^\perp} \,, 
  \]
  where $\sigma =\left(
  \begin{smallmatrix}
    1 & 0 \\ 0 & -1 
  \end{smallmatrix}\right)$. 
  The orbits on~$\KK^4$ are represented by~$0$, $b_0$,
  $b_3$, those on $(\Sk4)/T$ are represented by $T$,
  $\s02+T$, $\s03+T$, $\s13+T$, and those on $(\Sk4)/T^\perp$ by
  $T^\perp$, $\s03+T^\perp$, $\s23+T^\perp$. 
\end{prop}
\begin{proof}
  The intersection of $T$ with the Klein quadric is the single
  point~$\smallspnK{\s01}$. Thus the stabilizer of $T$ leaves the line
  $\smallspnK{b_0,b_1}$ invariant, and
  \[
  (\GL4\KK)_T \le \set{\left(
      \begin{array}{cc}
        A & B \\
        0 & C 
      \end{array}\right)}{A,C\in\GL2\KK, B\in\KK^{2\times2}} \,.  
  \]
  Evaluating the requirement
  \[
  \left(
      \begin{array}{cc}
        A & B \\
        0 & C 
      \end{array}\right) . (\s03+\s12)  =
    \left(
      \begin{array}{cc}
        A & B \\
        C & D
      \end{array}\right)
    \in  \smallspnK{\s01,\s03+\s12} 
  \]
  gives the condition $C\in\KK^\times \sigma A\sigma$, as claimed.
  In order to see that the orbits of $\s03+T^\perp$ and $\s23+T^\perp$
  are large enough, it is helpful to observe $\left(
    \begin{smallmatrix}
      A & 0 \\
      0 & \sigma A\sigma 
    \end{smallmatrix}\right)
  .\s03 \in (\det A)\s03+T^\perp$,
  and to use $B=\left(
    \begin{smallmatrix}
      f & 0 \\
      0 & 0 
    \end{smallmatrix}\right)\sigma A\sigma$. 
\end{proof}

\begin{prop}\label{stabS}
  The stabilizer $(\GL4\KK)_S$ of\/ $S$ is
  \[
  (\GL4\KK)_S = \set{\left(
      \begin{array}{cc}
        A & 0 \\
        0 & B 
      \end{array}\right)}{A,B\in\GL2\KK}
  \cup \set{\left(
      \begin{array}{cc}
        0 & A \\
        B & 0 
      \end{array}\right)}{A,B\in\GL2\KK}
  = (\GL4\KK)_{S^\perp} \,. 
  \]
  The orbits on~$\KK^4$ are represented by~$0$, $b_0$,
  $b_0+b_2$, those on $(\Sk4)/S$ are represented by $S$,
  $\s02+S$, $\s02+\s13+S$, and those on $(\Sk4)/S^\perp$ by $S^\perp$,
  $\s01+S^\perp$, $\s01+\s23+S^\perp$. 
\end{prop}
\begin{proof}
  The description of the stabilizer follows immediately from the
  observation that $S$ meets the Klein quadric in exactly two points,
  corresponding to the two lines $\smallspnK{b_0,b_1}$ and
  $\smallspnK{b_2,b_3}$, respectively.
\end{proof}

\begin{prop}\label{stabJF}
  The stabilizer $(\GL4\KK)_{J(F)}$ of\/ $J(F)=\smallspnK{\s12,\s13,\s23}$ coincides with the stabilizer of
  the plane $(\smallspnK{b_1,b_2,b_3})$. Thus
  \[
  (\GL4\KK)_{J(F)} = \set{\left(
      \begin{array}{cc}
        a & 0 \\
        b & C 
      \end{array}\right)}{%
      a\in\KK\setminus\{0\},  b\in\KK^3,  C\in\GL3\KK} \,. 
  \]
  The orbits on~$\KK^4$ are represented by~$0$, $b_1$,
  $b_0$, and those on $(\Sk4)/J(F)$ are represented by $J(F)$,
  $\s01+J(F)$. 
\end{prop}
\begin{proof}
  This follows from the fact (see~\ref{genericOrbits}.\ref{Mthree})
  that $J(F)$ corresponds to the set of all lines in the
  plane~$P=\smallspnK{b_1,b_2,b_3}$. 
\end{proof}

\begin{prop}\label{stabET}
  The stabilizer of\/ $E+T$ is
  \[
  (\GL4\KK)_{E+T} = \set{\left(
      \begin{array}{ccc}
        a & b'& c \\
        0 & D & e \\
        0 & 0 & \frac1a\det D
      \end{array}\right)}{\begin{array}{cc}
      a,c\in\KK, & a\ne0, \\ b,e\in\KK^2, & D\in\GL2\KK
    \end{array}}\,. 
  \]
  The orbits on~$\KK^4$ are represented by~$0$, $b_0$,
  $b_1$, $b_3$,  and the orbits on $(\Sk4)/(E+T)$ are represented by $E+T$,
  $\s03+(E+T)$, $\s13+(E+T)$. 
\end{prop}
\begin{proof}
  The intersection of $E+T$ with the Klein quadric is just~$E$. Thus
  $(\GL4\KK)_{E+T}$ is contained in $(\GL4\KK)_E$,
  see~\ref{stabE}. Evaluating the condition 
  \[
  \left(
    \begin{array}{ccc}
        a & b'& c \\
        0 & D & e \\
        0 & 0 & f
      \end{array}\right).(\s03+\s12) =
  \left(
    \begin{array}{ccc}
        a & b'& c \\
        0 & D & e \\
        0 & 0 & f
      \end{array}\right)
  \left(
    \begin{array}{ccc}
        0 & 0 & 1 \\
        0 & i & 0 \\
       -1 & 0 & 0
      \end{array}\right)
  \left(
    \begin{array}{ccc}
        a & 0 & 0 \\
        b & D'& 0 \\
        c & e'& f
      \end{array}\right)
    \in E+T 
  \]
  we obtain the description of the stabilizer.   
\end{proof}

\begin{prop}\label{stabES}
  The stabilizer $(\GL4\KK)_{E+S}$ of\/ $E+S$ is
  \[
  (\GL4\KK)_{E+S} = \set{\left(
      \begin{array}{cccc}
       a & b & 0 & c \\
       0 & d & 0 & 0 \\
       0 & e & f & g \\
       0 & 0 & 0 & h
      \end{array}\right)}{\begin{matrix}
           a,d,f, h \in \KK^{\times} \\ 
           b, c, e, g \in \KK
         \end{matrix}}
       \cup 
      \set{\left(
      \begin{array}{cccc}
       0 & c & a & b \\
       0 & 0 & 0 & d \\
       f & g & 0 & e \\
       0 & h & 0 & 0
      \end{array}\right)}{\begin{matrix}
       b, d, e, h \in \KK^{\times} \\ 
       a, c, f, g \in \KK
      \end{matrix}} \,.
  \]
  This group has $5$ orbits on~$\KK^4$, represented by~$0$, $b_0$,
  $b_0+b_2$, $b_0+b_2+b_3$, $b_1+b_3$,  and $4$ orbits on $(\Sk4)/(E+S)$, represented by $E+S$,
  $\s03+(E+S)$, $\s13+(E+S)$, $\s03+\s12+(E+S)$. 
\end{prop}
\begin{proof}
  The plane $E+S$ meets the Klein quadric in two lines, namely $E$ and
  $\smallspnK{\s02,\s23}$.
  Thus the stabilizer fixes their intersection point $\smallspnK{\s02}$ which corresponds
  to the line~$\smallspnK{b_0,b_2}$ while the planes $\smallspnK{b_0,b_1,b_2}$ and
  $\smallspnK{b_0,b_2,b_3}$ are either swapped or left invariant.
  Now easy calculations yield the stabilizer. 
\end{proof}

Instead of $T+S$, we consider a different representative of the orbit
$\GL4\KK.(T+S)$, namely 
$K:= \smallspnK{ \s02, \s03+\s12, \s13 }=\set{{\left(
      \begin{smallmatrix}
        0 & X \\
       -X & 0
      \end{smallmatrix}\right)}}{X\in\KK^{2\times2}, X'=X}$.
Note that the orthogonal spaces $(T+S)^\perp$ and
$K^\perp=\smallspnK{\s01,\s03-\s12,\s23}$ belong to the orbit of $T+S$, as well. 

\begin{prop}\label{PhiDelta}\label{stabTS}
  The stabilizer $(\GL4\KK)_{T+S}$ of\/ $T+S$ is a conjugate of
  $(\GL4\KK)_{K} = \Phi\Delta$, where 
  \[
  \Phi := \set{\left(\begin{array}{cccc}
            a & 0 & b & 0 \\
            0 & a & 0 & b \\
            c & 0 & d & 0 \\
            0 & c & 0 & d
           \end{array}\right)}{\left(\begin{array}{cc}
                           a & b \\
                           c & d
                          \end{array}\right) \in \GL2\KK} \,, 
  \quad 
  \Delta:=\set{\left(\begin{array}{cc}
              A & 0 \\
              0 & A
             \end{array}\right)}{A \in \GL2\KK} \,. 
  \]
  Note that $\Phi$ and $\Delta$ centralize each other, their
  intersection is the center of $\GL4\KK$, and each of these factors
  is isomorphic to~$\GL2\KK$.

  The orbits of $\Phi\Delta$ on $\KK^4$ are represented by~$0$, $b_0$,
  $b_0+b_3$.  In order to understand the orbits on $(\Sk4)/K$, we have
  to consider the multiplicative group $\KK^\sqt$ of squares, 
  pick a set~$R_*$ of representatives for the group
  $\KK^\times/\KK^\sqt$ of square classes, and distinguish the
  cases:
  \begin{enum}
  \item If $\Char\KK\ne2$ then the orbits on $(\Sk4)/K$ are
    represented by the elements of the set
    \[
    \{0\}\cup\set{\s02+r\s13+K}{r\in R_*\cup\{0\}} \,.
    \]
  \item If $\Char\KK=2$ we also need a set $R_\wp$ of coset
    representatives for the additive group $\KK/\wp$,
    cf.~\ref{Arf}. The orbits on $(\Sk4)/K$ are represented by the
    elements of the set
    \[
    \{0\}\cup\set{\s02+r\s13+K}{r\in R_*\cup\{0\}}
         \cup\set{\s01+\s03+c\s23}{c\in R_\wp} \,.
    \]
  \end{enum}
\end{prop}
\begin{proof}
  First of all, we remark that
  \[
  \left(
    \begin{array}{cccc}
      1 & 0 & 0 & 0 \\
      0 & 0 &-1 & 0 \\
      0 & 1 & 0 & 0 \\
      0 & 0 & 0 & 1 \\
    \end{array}\right)
  \quad\text{ and }\quad
  \left(
    \begin{array}{cccc}
      1 & 0 & 0 & 0 \\
      0 & 1 & 0 & 0 \\
      0 & 0 &-1 & 0 \\
      0 & 0 & 0 & 1 \\
    \end{array}\right)
  \]
  map $T+S$ to~$K$ and~$K^\perp$, respectively. 
  It is easy to check that $\Phi$ and $\Delta$ are subgroups of
  $(\GL4\KK)_K$. The plane $K$ intersects the Klein quadric in the 
  non-degenerate conic
  \[
  \mathcal C:=\set{\spnK{\left(
      \begin{smallmatrix}
        0 & X \\
       -X & 0
     \end{smallmatrix}\right)}}{X'=X\ne0,\det X=0}
  \,.
  \]

  The group $\Phi\Delta$ acts on~$\mathcal C$ via
  \[
  \left(
    \begin{matrix}
      a & b \\ c & d 
    \end{matrix}\right) 
  \left(
    \begin{matrix}
      A & 0 \\ 0 & A 
    \end{matrix}\right) . \left(
    \begin{matrix}
      0 & X \\
      -X & 0 
    \end{matrix}\right) =
  (ad-bc) \left(
    \begin{matrix}
      0 & AXA' \\
      -AXA' & 0 
    \end{matrix}\right). 
  \]
  It is easy to see that this action of $\Delta$ on $\mathcal C$ is
  sharply $3$-transitive; in fact, we have given here an explicit
  representation of $\PGL2\KK$ as the group of similitudes of the
  quadratic form $\det$ on the space of symmetric $2\times2$ matrices
  yielding a sharply $3$-transitive action on the conic. 

  The stabilizer of the three points $\smallspnK{\s02}$, $\smallspnK{\s13}$, and
  $\smallspnK{\s02+\s03+\s12+\s13}$ in $(\GL4\KK)_K$ is the group~$\Phi$.
  Thus we have $(\GL4\KK)_K=\Phi\Delta$, as claimed.

  In order to understand the action of $\Phi\Delta$ on~$(\Sk4)/K$, we
  first describe quadratic forms on $\KK^2$ by upper matrices:
  For $X:=\left(
    \begin{smallmatrix}
      r & s \\ t & u 
    \end{smallmatrix}\right)$
  we put $q_X\colon v\mapsto v'Xv$ and $\widehat{X}:=\left(
    \begin{smallmatrix}
      r & s+t \\ 0 & u 
    \end{smallmatrix}\right)$.
  Clearly we have $q_X=q_{\widehat{X}}$, and a basis transformation
  $v\mapsto Bv$ transforms $q_X$ to $q_{B'XB}$.

  For the cosets in $(\Sk4)/K$ we use representatives of the form
  \[
  \rho\left(
    \begin{smallmatrix}
      x & y_1 \\ y_2 & z 
    \end{smallmatrix}\right)
  :=
  x\s01+(y_1+y_2)\s03+z\s23
  =\left(
    \begin{array}{cccc}
      0 & x & 0 & y_1+y_2 \\
     -x & 0 & 0 & 0 \\
      0 & 0 & 0 & z \\
     -y_1-y_2 & 0 &-z & 0 
    \end{array}\right) \,.
  \]
  Note that $\rho(X)=\rho(\widehat{X})$ holds for each $X\in\KK^{2\times2}$.
  A straightforward computation gives
  $(B,A).\widehat{X}-\det(A)\,\widehat{BXB'} \in K$ for each pair
  $(B,A)\in\Phi\times\Delta$.  Thus the action of $\Phi\Delta$ on
  $(\Sk4)/K$ is equivalent to the action on the space of quadratic
  forms on~$\KK^2$.

  If $\Char\KK\ne2$ every quadratic form $q_X$ is diagonalizable: i.e., there
  exists $B\in\GL2\KK$ with $\widehat{BXB'}=\left(
    \begin{smallmatrix}
      r & 0 \\ 0 & u 
    \end{smallmatrix}\right)$
  with $r,u\in\KK$. If the form is non-zero, we may also assume
  $r\ne0$. Replacing the second basis vector by a scalar multiple just
  multiplies $u$ by the square of that factor. Choosing $A$ such that
  $\det A=r^{-1}$ gives the assertion for~$\Char\KK\ne2$. 

  If $\Char\KK=2$ we have to distinguish between diagonalizable forms
  (where the equivalence classes are again described by the square
  classes) and non-diagonalizable forms.  Since the action of
  $\Phi\Delta$ allows to pick representatives that assume the
  value~$1$, the latter orbits are characterized by the Arf invariant,
  see~\ref{nonDiagonalCharTwo}. This yields the assertion for $\Char\KK=2$.  
\end{proof}

  For the case where $\Char\KK\ne2$, we obtain an alternative
  description of $(\GL4\KK)_{T+S}$ in~\ref{splitQuaternions}  below,
  in terms of the split quaternion algebra.

  \begin{rema}
    For a finite field ~$\FF$ of characteristic~$2$, we have
    $\left|\KK/\wp(\KK)\right|=2$. Anything may happen in the infinite
    case, even if $\KK$ is a perfect field. 
  \end{rema}

\section{Examples involving field extensions}
\label{sec:fieldextensions}

Subspaces that meet the Klein quadric in only few points (or even no
points at all) present particular problems when classifying them or
when determining the automorphism groups of the corresponding
Heisenberg algebras. Therefore, we will now discuss connections
between anisotropic subspaces, quadratic extension fields, and (in
Section~\ref{sec:quaternions}) quaternion fields, together with 
constructions of Heisenberg algebras related to these structures.

\begin{exam}
  An explicit model for the (unique) isomorphism type of reduced
  Heisenberg algebra $\gheis VZ\beta$ with \mbox{$\dim V=2$} and $\dim
  Z=1$ is $\gheis{\KK^2}{\KK}{\det}$, where $\det(v,w)$ is the usual
  determinant of the $2\times2$ matrix with columns $v,w$. This
  algebra is \emph{the} Heisenberg algebra used to explain the
  uncertainty principle.
\end{exam}

\begin{exam}\label{squared}
  Among the reduced Heisenberg algebras $\gheis VZ\beta$ with $\dim V=4$ and
  $\dim Z=2$, we find the direct product 
  $\gheis{\KK^2}{\KK}{\det}\times\gheis{\KK^2}{\KK}{\det}$.
  This algebra is isomorphic to $\gheis{\KK^4}{\KK^2}{\beta}$, where the
  kernel of $\hat\beta$ is~$S^\perp=\smallspnK{\s02,\s03,\s12,\s13}$,
  see~\ref{ETSdef}. 
\end{exam}

\begin{ndef}[Examples from Quadratic Extensions]\label{complexGH}
  Let $\LL=\KK(u)$ be a quadratic extension field of~$\KK$, where the
  minimal polynomial of~$u$ over~$\KK$ is $X^2+tX+d$.  As a
  $\KK$-algebra, the Heisenberg algebra $\gheis{\LL^2}\LL\det$ is then
  isomorphic to $\gheis{\KK^4}{\KK^2}{\beta_u}$, where the kernel
  of~$\widehat{\beta_u}$ belongs to the orbit of $(\formspace
  1dt)^\perp$, cf.~\ref{orbitsFourSpecial}.

  Note that we may choose $t\in\{0,1\}$, and that $t=1$ is only needed
  if $\Char\KK=2$ and the extension is a separable one
  (cf.~\cite[8.11, p.\,313]{MR770063}).

  Here is a more explicit way to describe the action of $\GL2\LL$ on
  $\ker{\widehat{\beta_u}}$:
  We identify $\LL$ with $\set{\left(
      \begin{smallmatrix}
        x & -yd \\
        y & x+yt 
      \end{smallmatrix}\right)}{x,y\in\KK}$ 
  and $\LL^{2\times2}$ with the set $\set{\left(
      \begin{smallmatrix}
        A & B \\
        C & D
      \end{smallmatrix}\right)}{A,B,C,D\in\LL}$ of block matrices.
  Thus $u$ is identified with
  $\left(
      \begin{smallmatrix}
        0 & -d \\
        1 & t
      \end{smallmatrix}\right)
  =\left(
      \begin{smallmatrix}
        0 & -d \\
        1 & -t
      \end{smallmatrix}\right)
    $.
    
  For $\delta:=\left(
    \begin{smallmatrix}
      d & 0 \\
      0 & -1
    \end{smallmatrix}\right)$ we observe
  $\delta A'=A\delta$ for all $A\in\LL$.
  Now the set
  \[
  P_\LL^{}:=\set{\left(
      \begin{smallmatrix}
        0 & X\delta \\
        -X\delta & 0 
      \end{smallmatrix}\right)}{X\in\LL}
  = \spnK{d\s02-\s13,d\s03+d\s12-t\s13}
  \]
  is a subspace of $\Sk4$, belongs to the orbit of $\formspace1dt$ and
  is invariant under the action of $\GL2\LL\le\GL4\KK$.
  Using $i=\left(
    \begin{smallmatrix}
      0 & 1 \\
      -1 & 0 
    \end{smallmatrix}\right)$ we obtain the orthogonal space as 
  \[
  P_\LL^\perp=\set{\left(
      \begin{smallmatrix}
        xi & Yi \\
        -(Yi)' & zi 
      \end{smallmatrix}\right)}{x,z\in\KK,Y\in\LL}
  =\smallspnK{\s01,d\s02+\s13+t\s12,\s03-\s12,\s23} \,.
  \]
  Of course the orthogonal space $P_\LL^\perp$ is also
  invariant under $\GL2\LL$; this space serves as a model for
  $\ker{\widehat{\beta_u}}$.
  For $X\in\LL$ we observe $iX'i^{-1} = \widebar{X}$, where $\widebar{r+su}
  := r-su +st$.
  This helps to verify
  invariance of $P_\LL^\perp$ explicitly; here we use the fact that
  $t\ne0$ only occurs if $\Char\KK=2$.

  The intersection of the projective
  $3$-space coordinatized by $P_\LL^\perp$ with the
  Klein quadric is the ellipsoid 
  \[
  \mathcal O:=
  \set{\spnK{\left(
        \begin{matrix}
          xi & Yi \\
          -(Yi)' & zi
        \end{matrix}\right)}}{x,z\in\KK,Y\in\LL,Y\widebar{Y}=xz E_2} \,. 
  \]
  It is easy to see that $\left(
    \begin{smallmatrix}
      0 & 1 \\
      1 & 0 
    \end{smallmatrix}\right)\in\GL2\LL$ interchanges
  $\s01=\left(
    \begin{smallmatrix}
      i & 0 \\0 & 0 
    \end{smallmatrix}\right)$
  with $\s23=\left(
    \begin{smallmatrix}
      0 & 0 \\
      0 & i
    \end{smallmatrix}\right)$,
  and that the orbit $\set{\spnK{\left(
      \begin{smallmatrix}
        B\widebar{B} i & Bi \\
        -(Bi)' & i 
      \end{smallmatrix}\right)}}{B\in\LL}$ of $\spnK{\s23}$ under the
  stabilizer 
  $\set{\left(
      \begin{smallmatrix}
        A & B \\
        0 & D
      \end{smallmatrix}\right)}{A,D\in\LL^\times, B\in\LL}$
  of $\spnK{\s01}$ in $\GL2\LL$ equals
  $\mathcal O\setminus\{\spnK{\s01}\}$. 
\end{ndef}

\begin{coro}\label{understandP3}
  The representatives in $\mathcal P_1$ may be chosen in the
  form~$P_\LL^\perp$, and those in $\mathcal P_3$ may be chosen from
  the collection of spaces of the form
  $P_\LL^0:=P_\LL^{}\oplus\smallspnK{\s01}$. Here $\LL$ and $P_\LL^{}$ are
  constructed as in~\ref{complexGH} using an irreducible polynomial
  $X^2+tX+d$ over~$\KK$ with $2t=0$.  
\qed
\end{coro}

\begin{lemm}\label{PLLperpExplicitly}
  Let $X^2+tX+d$ be irreducible over~$\KK$, with $2t=0$.
  As in~\ref{complexGH} we choose $\LL$ and $P_\LL^{}$ and write $i:=\left(
    \begin{smallmatrix}
      0 & 1 \\
      -1 & 0
    \end{smallmatrix}\right)$ and $\delta:=\left(
    \begin{smallmatrix}
      d & 0 \\
      0 & -1
    \end{smallmatrix}\right)$. 
  \begin{enum}
    \item If the extension\/ $\LL/\KK$ is separable then $\KK^{2\times2}=\LL i
      \oplus\LL\delta$.
      In this case the orthogonal space $P_\LL^\perp 
      $ is a vector
      space complement to $P_\LL^{}$ in $\Sk{4}$. 
    \item\label{PLLperpExplicitlyInsepCase}
      If the extension\/ $\LL/\KK$ is inseparable then $\LL i =
      \LL\delta$. In this case the polar form vanishes on~$P_\LL^{}$, and
      $P_\LL^{}\subset P_\LL^\perp$.
    \item\label{tZeroCase} If\/ $t=0$ then $\KK^{2\times2}=\LL\oplus\LL\delta$.
    \item If\/ $t=1$ then $\KK^{2\times2}=\LL\oplus\LL i$.
    \item If\/ $d=1=t$ then $\LL=\LL\delta$.
    \item If\/ $d=1$ and $t=0$ then $\LL=\LL i$.
  \end{enum}
\end{lemm}
\begin{proof}
  For $X:=\left(
    \begin{smallmatrix}
      a & -bd \\
      b & a+bt
    \end{smallmatrix}\right)$ and $Y:=\left(
    \begin{smallmatrix}
      x & -yd \\
      y & x+yt
    \end{smallmatrix}\right)$ 
  in~$\LL$ we compute $Xi=\left(
    \begin{smallmatrix}
      bd & a \\
      -a-bt & b
    \end{smallmatrix}\right)$ and $Y\delta=\left(
    \begin{smallmatrix}
      xd & yd \\
      yd & -x-yt
    \end{smallmatrix}\right)$.
  Equality $Xi=Y\delta$ thus implies $b=x$, $a=yd$ and then $-bt=2a$,
  $2x=-yt$. If the extension is separable, we have $\Char\KK\ne2$ or
  $t\ne0$. In either case, we infer $A=0=X$.  This means $\LL
  i\cap\LL\delta=\{0\}$. Using $-(Xi)'=iX'=\widebar{X}i$ and
  $-(Y\delta)'=-\delta Y'=-Y\delta$ we infer $\Sk{4} = \set{\left(
          \begin{smallmatrix}
            ai & Xi \\
            \widebar{X}i & ci 
          \end{smallmatrix}\right)}{a,c\in\KK, X\in\LL} \oplus
      \set{\left(
          \begin{smallmatrix}
            0 & Y\delta \\
            -Y\delta & 0 
          \end{smallmatrix}\right)}{Y\in\LL}
      = P_\LL^\perp\oplus P_\LL^{}$ as claimed. 

  The extension $\LL/\KK$ is inseparable if $\Char\KK=2$ and $t=0$.
  In this case the sets in question coincide: $\LL i = \set{\left(
      \begin{smallmatrix}
        bd & a \\
        a & b
      \end{smallmatrix}\right)}{a,b\in\KK} = \LL\delta$. 
\end{proof}

\begin{theo}\label{complexGHaut}
  Let\/ $X^2+tX+d$ be irreducible over~$\KK$, with $2t=0$.
  We choose $\LL$ and\/ $P_\LL^{}$ as in~\ref{complexGH} and put
  $P_\LL^0:=P_\LL^{}\oplus\smallspnK{\s01}$.
  Moreover, we write $\xi:=\left(
    \begin{smallmatrix}
      1 & t \\
      0 & -1
    \end{smallmatrix}\right)$
  and\/ $\Xi:=\left(
    \begin{smallmatrix}
      \xi & 0 \\
      0 & \xi
    \end{smallmatrix}\right)\in\GL4\KK$.
  Then conjugation by $\xi$ induces the (possibly trivial) generator
  of\/ $\Gal{\LL/\KK}$ on~$\LL$, and 
  \begin{align*}
    (\GL4\KK)_{P_\LL^{}} &= 
    \GL2\LL \left<\Xi\right> = (\GL4\KK)_{P_\LL^\perp}
    \text{ where }
    \GL2\LL = 
    \set{\left(
        \begin{array}{cc}
          A & B \\
          C & D 
        \end{array}\right)}{
      \begin{array}{c}
        A,B,C,D\in\LL,\\AD-BC\ne0
     \end{array}} \,, \\
    (\GL4\KK)_{P_\LL^0} &= \set{\left(
        \begin{array}{cc}
          A & B \\
          0 & D 
        \end{array}\right)}%
    {A,D\in\LL^\times,B\in\KK^{2\times2}}
    \left<\Xi\right> 
     \,.
  \end{align*}
  The group $(\GL4\KK)_{P_\LL^{}}$ acts with~$2$ orbits on~$\KK^4$,
  represented by~$0$ and\/~$b_0$, and with~$2$ orbits on
  $(\Sk4)/P_\LL^\perp$, represented by $P_\LL^{}$ and any other coset
  of~$P_\LL^{}$.  %
  The orbits on $(\Sk4)/P_\LL^{}$ are more
  complicated, we have to distinguish cases:
  \begin{enum}
  \item\label{sepCase}
    If the extension $\LL/\KK$ is separable (in particular, if\/
    $\Char\KK\ne2$) then the action of\/ $\GL2\LL$ on $(\Sk4)/P_\LL^{}$
    is equivalent to the action on the space of hermitian $2\times2$
    matrices, cf.~\ref{HermiteEq}. Here $\Xi$ acts as $-\id$ on the
    set of diagonal matrices, and an orbit under
    $\GL2\LL\left<\Xi\right>$ is the union of two different\/
    $\GL2\LL$-orbits if, and only if, the norm group of the
    corresponding quaternion field does not contain~$-1$. 
  \item\label{complexGHautInsepCase} Now assume that\/ $\LL/\KK$ is an inseparable
    extension. Then the action of\/ $\GL2\LL$ on
    $P_\LL^\perp/P_\LL^{}$ is equivalent to the action of\/
    $\GL2\LL^\sq\le\GL2\KK$ on~$\KK^2$.  The orbits under that action
    are represented by the elements of\/ $\smallset{(r,0)'}{r\in
      R_{\KK/\LL^\sq}^{}} \cup R^{(2)}_{\KK/\LL^\sq}$  
    (see~\ref{diagonalCharTwo}).

  The $\GL2\LL$-orbits on
  $\left((\Sk4)/P_\LL^{}\right)\setminus\left(P_\LL^\perp/P_\LL\right)$
  are represented by the elements of the set\/
  $\set{\rho_z+P_\LL^{}}{z\in R_{\actSL}}$ where $\rho_z:=\left(
    \begin{smallmatrix}
      i & E_2 \\
      E_2 & zi 
    \end{smallmatrix}\right)$ and\/ $R_{\actSL}$ is a set of
  representatives for the orbits under the action
  \[
  \actSL\colon \SL2\LL\times\KK^2 \to \KK^2 \colon
  \left( \left(
      \begin{array}{cc}
        A & B \\
        C & D
      \end{array}\right),
    \left(
      \begin{array}{c}
        x_1 \\
        x_2
      \end{array}\right)
  \right)
  \mapsto
  \left(
    \begin{array}{cc}
      A^2x_1+B^2x_2 + (AB+(AB)')i \\
      C^2x_1+D^2x_2 + (CD+(CD)')i 
    \end{array}
  \right) \,.
  \]
  Note that no point is fixed under this action of\/ $\SL2\LL$ by
  affine transformations of $\KK^2$ (viewed as an affine space over
  $\LL^\sq\le\KK$). 
\item In any case (separable or not), the orbits of
  $(\GL4\KK)_{P_\LL^0}$ on $\KK^4$ are represented by $0$, $b_0$,
  $b_2$.
\item\label{Case0} %
  For the orbits on $({\Sk4})/P_\LL^0$ we pick a set $R_{N}
  \subseteq\KK^\times$ of representatives for the cosets modulo
  $N_{\LL/\KK}(\LL^\times)\left<-1\right>$. Then the orbits on
  $(\Sk4)/P_\LL^0$ are represented by the elements of\/ $\{P_\LL^0,
  (\s02+\s13)+P_\LL^0\} \cup \set{c\s23}{c\in R_{N}}$ --- again,
  irrespective of separability.
  \end{enum}
\end{theo}
\begin{proof}
  Clearly the group $\GL2\LL=\set{\left(
      \begin{smallmatrix}
        A & B  \\
        C & D  \\
      \end{smallmatrix}\right)\in\GL4\KK}{%
    A,B,C,D\in\LL,\,AD-BC\ne0 }$
  is contained in $(\GL4\KK)_{P_\LL^{}}$ and acts transitively
  on the intersection $\mathcal O$ of the Klein quadric with the
  projective $3$-space coordinatized by $P_\LL^{}\setminus\{0\}^\perp$,
  cf.~\ref{complexGH}.
  Thus it remains to determine the
  stabilizer of $\smallspnK{\s01}$ in $(\GL4\KK)_{P_\LL^{}}$.
  Evaluating the condition 
  \[
  \left(
    \begin{array}{cc}
      A & B \\
      0 & D 
    \end{array}\right) 
  \left(
    \begin{array}{cc}
      0 & X\delta \\
      -X\delta & 0 
    \end{array}\right)
  \left(
    \begin{array}{cc}
      A' & 0 \\
      B' & D' 
    \end{array}\right) 
  =
  \left(
    \begin{array}{cc}
      AX\delta B'-BX\delta A' & AX\delta D' \\
      -DX\delta A' & 0 
    \end{array}\right)
  \in P_\LL^{}
  \]
  we find that for each $X\in\LL^\times$ there exists $L_X\in\LL^\times$ such that
  $AX\delta D'=L_X\delta$.  
  Specializing $X=1\in\LL$ we obtain $\delta D'=A^{-1}L_1\delta$ and
  then $AXA^{-1}=L_X L_1^{-1}\in\LL$. Thus $A$ belongs to the
  normalizer of $\LL$ in $\GL2\KK$. According to Schur's Lemma
  (e.g., see~\cite[{3.5,~p.\,118}]{MR1009787}
  or \cite[{Ch.\,XVII, Prop.\,1.1}]{MR1878556}), this normalizer is 
  $\LL^\times\,\Gal{\LL/\KK}=\LL^\times\left<\xi\right>$.
  As $\Xi$ belongs to $(\GL4\KK)_{P_\LL^{}}$ we may assume $A\in\LL^\times$
  from now on. Then $D=\delta(A^{-1}L_1)'\delta^{-1}=A^{-1}L_1$ also
  lies in~$\LL^\times$. There remains the condition that $BX\delta A'$ is
  symmetric for each $X\in\LL$. Specializing $X=1$ and $X=u=\left(
    \begin{smallmatrix}
      0 & -d \\
      1 & t
    \end{smallmatrix}\right)$ we find $BA\in\LL$ and therefore $B\in\LL$.
  We have thus proved that
  $(\GL4\KK)_{P_\LL^{}}=\GL2\LL\,\left<\Xi\right>$.

  The stabilizer of ${P_\LL^0}$ in $(\GL4\KK)$ fixes $\smallspnK{\s01}$ because
  this is the only intersection point of the Klein quadric with the
  plane coordinatized by $P_\LL^0$. Thus $(\GL4\KK)_{P_\LL^0}$ is
  contained in $\set{\left(
      \begin{smallmatrix}
        A & B \\
        0 & D
      \end{smallmatrix}\right)}{A,D\in\GL2\KK, B\in\KK^{2\times2}}$.
  The elements of the stabilizer are characterized by the condition
  \[
    \left(
    \begin{array}{cc}
      AX\delta B'-BX\delta A' & AX\delta D' \\
      -DX\delta A' & 0 
    \end{array}\right)
  \in P_\LL^0 \,.
  \]
  As in the case discussed before, the upper right entry yields
  $A\in\LL^\times\left<\xi\right>$; we may assume $A\in\LL$, and then $D\in\LL$
  follows. However, the entry on
  the upper left does not mean any restriction now, and we obtain
  $(\GL4\KK)_{P_\LL^0}=\set{\left(
      \begin{smallmatrix}
        A & B \\
        0 & D
      \end{smallmatrix}\right)}{A,D\in\LL^\times,B\in\KK^{2\times2}}\left<\Xi\right>$,
  as claimed. 

  The assertions about orbits on~$\KK^4$ are easily verified for both
  stabilizers.  In order to understand the orbits on
  $(\Sk4)/P_\LL^\perp$ we pick a vector space complement $W$ for
  $P_\LL^\perp$ in $\Sk4$. In the separable case we may use $W = P_\LL^{}
  = \set{\left(
      \begin{smallmatrix}
        0 & X\delta \\
        -X\delta & 0 
      \end{smallmatrix}\right)}{X\in\LL}$
  while $W = \set{\left(
        \begin{smallmatrix}
          0 & X \\
          X' & 0
        \end{smallmatrix}\right)}{X\in\LL}$
  is a suitable choice in the inseparable case.
  In both cases it is easy to see that the action of $M\in\GL2\KK$ on~$W$ is
  given by multiplication of $X$ with $\det_\LL M$. Therefore, the 
  action on the set of non-trivial cosets in
  $(\Sk4)/P_\LL^\perp$ is transitive. 

  If the extension is separable then the action of
  $\GL2\LL\left<\Xi\right>$ on $(\Sk{4})/P_\LL^{}$ is equivalent to the
  action on the invariant orthogonal complement
  $P_\LL^\perp$. Using~\ref{PLLperpExplicitly} we see that this action
  is equivalent to the action on the space
  $\set{\left(
    \begin{smallmatrix}
      a & X \\
      \widebar{X} & c 
    \end{smallmatrix}\right)}{a,c\in\KK,X\in\LL}$
  of hermitian $2\times2$ matrices:
  \[
  \left(
    \begin{matrix}
      A & B \\
      C & D
    \end{matrix}\right) .
  \left(
    \begin{matrix}
      a & X \\
      \widebar{X} & c 
    \end{matrix}\right) :=  
  \left(
    \begin{matrix}
      A & B \\
      C & D
    \end{matrix}\right) 
  \left(
    \begin{matrix}
      a & X \\
      \widebar{X} & c 
    \end{matrix}\right) 
  \left(
    \begin{matrix}
      \widebar{A} & \widebar{C} \\
      \widebar{B} & \widebar{D}
    \end{matrix}\right) \,. 
  \]
  The orbit structure on $(\Sk4)/P_\LL^{}$ thus corresponds to the
  classification of hermitian forms over~$\LL$, cf.~\ref{HermiteEq}.
  An easy computation shows that $\Xi$ acts as $-\id$ on the diagonal
  forms. The assertion about fusion of $\GL2\LL$-orbits
  under~$\left<\Xi\right>$ now also follows from~\ref{HermiteEq}. 
  \goodbreak
  
  If $\LL/\KK$ is inseparable (i.e., if $\Char\KK=2$ and $t=0$) we
  note that $\set{\left(
      \begin{smallmatrix}
          xi & Y \\
          Y' & zi
        \end{smallmatrix}\right)}{x,z\in\KK,Y\in\LL}$
  is a vector space complement for $P_\LL^{}=\set{\left(
        \begin{smallmatrix}
          0 & \,Y\delta \\
          Y\delta\, & 0
        \end{smallmatrix}\right)}{Y\in\LL}$. 
    The action of $\GL2\LL$ on $P_\LL^\perp/P_\LL^{}$ is described by
    \[
    \left(
      \begin{matrix}
        A & B \\
        C & D 
      \end{matrix}\right) 
    \left(
      \begin{matrix}
        xi & 0 \\
        0 & zi
      \end{matrix}\right)
    \left(
      \begin{matrix}
        A' & C' \\
        B' & D'
      \end{matrix}\right) 
        = \left(
      \begin{matrix}
        (A^2x+B^2z)i & 0 \\
        0 & (C^2x+D^2z)i 
      \end{matrix}\right) \,.
    \]
    This looks like the usual action on diagonalizable quadratic forms
    (cf.~\ref{diagonalCharTwo}) 
    but with the group $\GL2\LL$ replacing $\GL2\KK$: if $x$ and $z$
    are linearly independent over~$\LL^\sq$ then the orbit of
    $(x,z)'$ consists of all bases for
    $\smallspn{x,y}{\LL^\sq}$. If $x$ and $z$ are linearly
    dependent over~$\LL^\sq$ then the orbit of $(x,z)'$ contains
    $(y,0)'$ where $\spn{y}{\LL^\sq} =
    \spn{x,z}{\LL^\sq}$. This gives the assertion about the orbits
    in $P_\LL^\perp/P_\LL^{}$. 

    For $A,B\in\LL$ there exists $b\in\KK$ with $B'-B=b i$, and both 
    $AB'-AB=A(B-B')=bAi$ and
    $AB'+BA'-(AB+(AB)') = b(A-\gal{A})i$ belong
    to~$\LL i$.  
    Thus $\left(
      \begin{smallmatrix}
        A & B \\
        C & D 
      \end{smallmatrix}\right) \in \GL2\LL$ maps 
    $\left(
      \begin{smallmatrix}
        xi & Y \\
        Y' & zi 
      \end{smallmatrix}\right) + P_\LL^{}$ to 
    \[
    \begin{array}{rcl}
      & &
    \left(
      \begin{array}{cc}
        (A^2x+B^2z)i+AYB+(AYB)' & (AD-BC)Y \\
        Y'(AD-BC)' & (C^2x+D^2z)i+CYD+(CYD)' 
      \end{array}\right) + P_\LL^{} \,.
    \end{array}
    \]
    From~\ref{PLLperpExplicitly}.\ref{tZeroCase}
    and~\ref{PLLperpExplicitly}.\ref{PLLperpExplicitlyInsepCase} we
    know $\KK^{2\times2}=\LL\oplus\LL\delta$ and %
    $\LL i = \LL\delta$. Thus each $\GL2\LL$-orbit on
    $\left((\Sk4)/P_\LL^{}\right)\setminus\left(P_\LL^\perp/P_\LL^{}\right)$
    contains a representative of the form $\rho_v+P_\LL^{}$ where
    $v=(v_1,v_2)\in\KK^2$ and $\rho_v:=\left(
    \begin{smallmatrix}
      v_1i & E_2 \\
      E_2 & v_2i 
    \end{smallmatrix}\right)$.
  Since $\rho_w+P_\LL^{} = \left(
    \begin{smallmatrix}
      A & B \\
      C & D
    \end{smallmatrix}\right)
  \rho_v+P_\LL^{}$ implies $AD-BC=1$, we are left with the action $\actSL$
  of $\SL2\LL$, as claimed. 

  It remains to determine the orbits of $(\GL4\KK)_{P_\LL^0}$ on
  $(\Sk4)/P_\LL^0$.  If $\LL/\KK$ is separable then
  $\KK^{2\times2}=\LL\delta\oplus\LL i$ and we may choose representatives for
  cosets modulo $P_\LL^0$ from the vector space complement
  $\set{\left(
      \begin{smallmatrix}
        0 & Xi \\
        iX' & ci 
      \end{smallmatrix}\right)}{c\in\KK,X\in\LL}$ to
  $P_\LL^0 = 
  \set{\left(
      \begin{smallmatrix}
        ai & X\delta \\
        -X\delta & 0 
      \end{smallmatrix}\right)}{a\in\KK, X\in\LL}$.
  If $X\ne0$ then the orbit of $\left(
      \begin{smallmatrix}
        0 & Xi \\
        iX' & 0
      \end{smallmatrix}\right)$
  contains $\s02+\s13 = \left(
      \begin{smallmatrix}
        0 & E_2 \\
        E_2 & 0
      \end{smallmatrix}\right)$.
  If $c\ne0$ then $\left(
    \begin{smallmatrix}
      E_2 & -c^{-1}X \\
      0 & D 
    \end{smallmatrix}\right) \in (\GL4\KK)_{P_\LL^0}$
  maps $\left(
      \begin{smallmatrix}
        0 & Xi \\
        iX' & ci
      \end{smallmatrix}\right) + P_\LL^0$
  to $\left(
      \begin{smallmatrix}
        0 & 0 \\
        0 & D\gal{D}ci
      \end{smallmatrix}\right) + P_\LL^0$. We may achieve
  $D\gal{D}c\in R_{N}^{}$ and assertion~\ref{Case0} follows from
  the fact that $\Xi \left(
      \begin{smallmatrix}
        0 & 0 \\
        0 & ci
      \end{smallmatrix}\right) \Xi'
    = \left(
      \begin{smallmatrix}
        0 & 0 \\
        0 & -ci
      \end{smallmatrix}\right)$.    

  If, finally, the extension is inseparable then $\LL\delta = \LL i$
  has trivial intersection with~$\LL$. Thus we may use the complement
  $\set{\left(
      \begin{smallmatrix}
        0 & X \\
       -X' & ci 
     \end{smallmatrix}\right)}{c\in\KK,X\in\LL}$.
 It is easy to see that $\s02+\s13$ belongs to the orbit of %
 $\left(
   \begin{smallmatrix}
     0 & X \\
     -X' & 0
   \end{smallmatrix}\right)$ if $X\in\LL^\times$.
 For $c\ne0$ we use $\left(
   \begin{smallmatrix}
     E_2 & c^{-1}Xi \\
     0 & D
   \end{smallmatrix}\right) \in (\GL4\KK)_{P_\LL^0}$
 in order to map $\left(
   \begin{smallmatrix}
     0 & X \\
     X' & ci \vphantom{\gal{D}}
   \end{smallmatrix}\right) + P_\LL^0$
 to $\left(
   \begin{smallmatrix}
     0 & 0 \\
     0 & D\gal{D}ci
   \end{smallmatrix}\right) + P_\LL^0$. %
 Now we may achieve $D\gal{D}c\in R_{N}^{}$.  This gives
 assertion~\ref{Case0} also in the inseparable case.
\end{proof}

\begin{rema}
  For the case $\Char\KK\ne2$ the assertion about
  $\Sigma_{\beta_u}=(\GL4\KK)_{P_\LL^{}}$ in~\ref{complexGHaut} is just a
  special case of~\cite[{Th.\,1.1.1}]{MR2410562}.  In fact, since
  automorphisms of the \emph{group} $\GHeis{\LL^2}{\LL}{\det}$ are
  considered in~\cite{MR2410562} the cited result yields all
  automorphisms of the \emph{Lie ring}
  $\gheis[\ZZ]{\LL^2}{\LL}{\det}$: apart from the automorphisms
  of~$\gheis[\KK]{\LL^2}{\LL}{\det}$ we also have the automorphisms
  induced by arbitrary field automorphisms of~$\LL$, and not only
  those from~$\Gal{\LL/\KK}$. Moreover, we have to add \emph{arbitrary}
  additive maps $\tau$ from~$\LL^2$ to~$\LL$, and not only the
  $\KK$-linear ones. See also~\ref{prob:fullAut}. 
\end{rema}

\goodbreak
\section{Examples involving quaternion algebras}
\label{sec:quaternions}
We will use quaternion algebras to describe the alternating maps
$\beta\colon\KK^4\times\KK^4\to\KK^3$ where
$\ker\widehat\beta\in\mathcal P_2$ describes an anisotropic plane,
having no point in common with the Klein quadric.  Our construction
and our results will not depend critically on $\Char\KK$ up to the
point where we investigate the action on~$(\Sk4)/\ker\widehat\beta$
in~\ref{orbitsWperp} and~\ref{orbitsModWCharTwo}.

\begin{ndef}[Anisotropic planes and quaternions]\label{def:WandHH}
We recall from~\ref{orbitsFourSpecial} that the orbit of any member
of~$\mathcal P_2$ contains a representative of the form
$\Formspace1cdt$ with $2t=0$.
A cyclic permutation of the basis vectors $b_0,b_1,b_2$ shows that this
orbit also contains 
\[
W:=W^t_{c,d}:=\spnK{\s03-\s12,c\s01-\s23,\s13+d\s02+t\s03} =
\set{\left(
    \begin{matrix}
      xci & Yi \\
      \gal{Y}i & -xi
    \end{matrix}\right)}{x\in\KK,Y\in\LL} 
\]
for $\LL:=\set{\left(
    \begin{smallmatrix}
      x & -yd \\
      y & x+yt 
    \end{smallmatrix}\right)}{x,y\in\KK}$. 
Since the restriction of the Pfaffian form to~$W$ is assumed to be
anisotropic the polynomial $X^2+tX+d$ is irreducible in~$\KK[X]$.  The
subalgebra $\LL$ of $\KK^{2\times2}$ is isomorphic to the 
extension field $\KK[X]/(X^2+tX+d)$,
cf.~\ref{complexGH}. We also recall that
conjugation by $\xi:=\left(
  \begin{smallmatrix}
    1 & t \\
    0 & -1
  \end{smallmatrix}\right)$
induces the (possibly trivial) generator of the Galois group
$\Gal{\LL/\KK}$ on~$\LL$, mapping $A:=\left(
    \begin{smallmatrix}
      x & -yd \\
      y & x+yt 
    \end{smallmatrix}\right)$ to $\gal{A}:=\left(
    \begin{smallmatrix}
      x+yt & yd \\
     -y & x
   \end{smallmatrix}\right)$.
We use the root $u:=\left(
  \begin{smallmatrix}
    0 & -d \\
    1 & t
  \end{smallmatrix}\right)$
of $X^2+tX+d$ in~$\LL$; note that $\gal{u}=t-u$.

  If the extension $\LL/\KK$ is separable then the
  hermitian matrix $\left(
    \begin{smallmatrix}
      1 & 0 \\
      0 & c
    \end{smallmatrix}\right)$ describes an anisotropic hermitian
  form~$h$ on~$ \LL^2$, and $\HH:=\HH_{\LL/\KK}^{c} = \set{\left(
      \begin{smallmatrix}
        A & -c\gal{B} \\
        B & \gal{A}
      \end{smallmatrix}\right)}{A,B\in\LL}$
  is a quaternion field, cf.~\ref{Hermite}.
  We identify $A\in\LL$ with $\left(
    \begin{smallmatrix}
      A & 0 \\
      0 & \gal{A}
    \end{smallmatrix}\right)$
  and put $I:=  \left(
    \begin{smallmatrix}
      0 & -cE_2 \\
      E_2 & 0
    \end{smallmatrix}\right)$;
  then $\HH=\LL\oplus I\LL$. 

  The matrices in $\HH$ may
  be considered as matrices for left multiplications $\lambda_a\colon
  x\mapsto ax$. In fact, with respect to the basis
  \[
  \left(
    \begin{matrix}
      E_2 & 0 \\
      0 & E_2
    \end{matrix}\right), \quad
  \left(
    \begin{matrix}
      u & 0 \\
      0 & \gal{u}
    \end{matrix}\right), \quad
  \left(
    \begin{matrix}
      0 & -cE_2 \\
      E_2 & 0
    \end{matrix}\right), \quad
  \left(
    \begin{matrix}
      0 & -c\gal{u} \\
      u & 0
    \end{matrix}\right) 
  \]
  we find that $\lambda_{A+IB}$ is described by $\left(
    \begin{smallmatrix}
      A & -c\gal{B} \\
      B & \gal{A}
    \end{smallmatrix}\right)$.
  Applying Schur's Lemma (cf.~\cite[{3.5,~p.\,118}]{MR1009787}
  or \cite[{Ch.\,XVII, Prop.\,1.1}]{MR1878556}) we infer that the centralizer
  $\HH^\opp$ of $\HH^\times$ in $\GL4\KK$ consists of the
  matrices for right multiplications $\rho_a\colon x\mapsto xa$, with
  respect to the same basis. A straightforward computation yields that 
  $\rho_{A+IB}$ is described by the matrix $\left(
    \begin{smallmatrix}
      A & -cB\xi \\
      B\xi & A
    \end{smallmatrix}\right)$.
  Thus $\HH^\opp = \set{\left(
      \begin{smallmatrix}
        A & -cB\xi \\
        B\xi & A
      \end{smallmatrix}\right)}%
  {(A,B)\in\LL^2\setminus\{(0,0)\}}$.

  If\/ $\LL/\KK$ is inseparable then $\xi$ is the identity matrix, and
  $\HH:=\set{\left(
      \begin{smallmatrix}
        A & -c{B} \\
        B & {A}
      \end{smallmatrix}\right)}{A,B\in\LL}$
  becomes a commutative field which (again by Schur's Lemma) coincides
  with its centralizer, indeed $\HH^\times=\HH^\opp$ in that
  case.

  In any case, we define $\tilde{a}:=\left(
    \begin{smallmatrix}
      \gal{A} & c\gal{B} \\
      -B & A 
    \end{smallmatrix}
  \right)$
  for $a=\left(
    \begin{smallmatrix}
      {A} & -c\gal{B} \\
      B & \gal{A} 
    \end{smallmatrix}
  \right)$
  and obtain an anti-automorphism of~$\HH$. Note that this
  anti-automorphism is the identity if $\LL/\KK$ is inseparable.
  We use the \emph{norm} $N\colon\HH\to\KK\colon a\mapsto
  \tilde{a}a$. 
\end{ndef}

\begin{defi}\label{def:PuHH}
  If $\LL/\KK$ is separable then  $\Pu{\HH} :=
  \smallset{X\in\HH}{\tilde{X}=-X} = 1^\perp$. For $\Char\KK=2$
  we have $1\in\Pu{\HH}$ and $\Pu{\HH}=\KK\oplus I\LL$.
  In the inseparable case the norm form on $\HH$ has trivial polar
  form, and $1^\perp=\HH$. 
  We extend the definition of $\Pu{\HH}$ to this case quite
  arbitrarily, putting $\Pu{\HH}:=\KK\oplus I\LL$ as in the remaining
  cases where $\Char\KK=2$.

  In any case, we find that the restriction $N|_{\Pu\HH}$ of the norm
  is equivalent to $-\pfaff|_W$. 
\end{defi}

\begin{lemm}\label{GOnormApplied}
  \begin{enum}
  \item If the extension\/ $\LL/\KK$ is separable then
    $\GO{}{N|_{\Pu{\HH}}} = \KK^\times\,\SO{}{N|_{\Pu{\HH}}} =
    \set{(x\mapsto sax\tilde{a})}{a\in\HH^\times,s\in\KK^\times}$ and\/
    $\SO{}{N|_{\Pu{\HH}}} = \set{(x\mapsto
      axa^{-1})}{a\in\HH^\times}$.
  \item In the inseparable case each one of the groups\/
    $\Orth{}{\pfaff|_W}$, $\Orth{}{N}$, $\Orth{}{N|_{\Pu{\HH}}}$ is
    trivial; note that $\Pu{\HH}=\KK\oplus I\LL$ by
    definition~\ref{def:PuHH}.  The groups $\GO{}{\pfaff|_W}$ and
    $\GO{}{N|_{\Pu{\HH}}}$ of similitudes coincide with
    $\KK^\times\,\id$ but $\GO{}{N}=\HH^\times$.
  \end{enum}
\end{lemm}
\begin{proof}
  For a separable extension $\LL/\KK$ the assertion has been
  proved in~\ref{fullGO3}, cf.~\ref{SkolemNoether}.
  
  Now assume that the extension $\LL/\KK$ is inseparable. For
  $x=x_0+x_1u+I(x_2+x_3u)$ with $x_0,x_1,x_2,x_3\in\KK$ the norm is
  given by $N(x)=x^2=x_0^2+dx_1^2+cx_2^2+cdx_3^2$. Thus it is
  diagonalizable and anisotropic, and the assertion follows
  from~\ref{diagAniso} and~\ref{onlySquareFactors}. 
\end{proof}

\begin{lemm}\label{quatInAnyChar}
  The space $W$ is invariant under $\HH^\times$ and under $\HH^\opp$.
  Both the action of\/~$\HH^\times$ on~$W$ and the action via
  $(a,X)\mapsto aX\tilde{a}$ on $\Pu\HH$ (see~\ref{def:PuHH}) are
  equivalent to that on $\KK\times\LL$ given by
  \[
  \left(
  \left(
    \begin{matrix}
      A & -c\gal{B} \\
      B & \gal{A}
    \end{matrix}\right),
  \left(x,Y\right)
  \right) 
  \mapsto
  \left(
    (A\gal{A}-cB\gal{B})x - ABY-\gal{ABY} \,,\, 
      2cA\gal{B}x + A^2Y - c\gal{B^2Y}
  \right) \,.
  \]
  Any element $\rho_a\in\HH^\opp$ induces the multiplication by its
  norm $\tilde{a}a$ on~$W$.
  Thus the action of $\HH^\times\HH^\opp$ on~$W$ is equivalent to an
  action by similitudes of the quadratic form $N|_{\Pu{\HH}}$. 
\end{lemm}
\begin{proof}
  Choose $j\in\LL\setminus\{0\}$ with $\gal{j}=-j$ in the separable
  case, and put $j:=1$ if $\LL/\KK$ is inseparable. A straightforward
  calculation shows that mapping $\left( x , Y\right)$ to $\left(
    \begin{smallmatrix}
      cxi & \phantom{-}Yi \vphantom{j} \\
      \gal{Y}i & -xi 
    \end{smallmatrix}\right) \in W$ or to 
  $\left(
    \begin{smallmatrix}
      xj & \phantom{-}{Y}j \\
      c^{-1}\gal{Y}j & -xj
    \end{smallmatrix}\right) \in\Pu\HH$,
  respectively, gives equivalences as claimed.  For the rest of the
  assertion, it remains to compute
  \[
  \left(
      \begin{matrix}
        A & -cB\xi \\
        B\xi & A
      \end{matrix}\right)
    \left(
      \begin{matrix}
        cxi & Yi \\
        \gal{Y}i & -xi
      \end{matrix}\right)
    \left(
      \begin{matrix}
        A' & (B\xi)' \\
        -(cB\xi)' & A'
      \end{matrix}\right)
    =
    (A\gal{A}+cB\gal{B})
    \left(
      \begin{matrix}
        cxi &  Yi \\
        \gal{Y}i &  -xi
      \end{matrix}\right)  \,.
\qedhere
  \]
\end{proof}

\goodbreak
\begin{lemm}\label{onlyTheIdentityRemains}
  If $M\in\GL4\KK$ fixes $b_0$ and induces a scalar multiple of\/
  $\id$ on~$W$ then $M=\id$. 
\end{lemm}
\begin{proof}
  We consider an element $M = \left(
    \begin{smallmatrix}
      A & B \\
      C & D
    \end{smallmatrix}\right)$ of the stabilizer of~$b_0$ in~$\GL4\KK$;
  then $A=\left(
    \begin{smallmatrix}
      1 & b \\
      0 & a 
    \end{smallmatrix}\right)$,
  $C=\left(
    \begin{smallmatrix}
      0 & x \\
      0 & y
    \end{smallmatrix}\right)$
  and $B,D\in\KK^{2\times2}$.  Moreover, we assume that $M$ induces
  $\alpha\,\id$ on~$W$.

  Evaluating the condition~(C1): 
  $M\left(
    \begin{smallmatrix}
      ci & 0 \\
      0 & i
    \end{smallmatrix}\right)M' = \left(
    \begin{smallmatrix}
      \alpha ci & 0 \\
      0 & \alpha i
    \end{smallmatrix}\right)$
  we find $\alpha=\det{D}$ and $B=\left(
    \begin{smallmatrix}
      z & w \\
      0 & 0 
    \end{smallmatrix}\right)$;
  then $a=\det{A}=\alpha$ follows.

  The condition~(C2): $M\left(
    \begin{smallmatrix}
      0 & i \\
      i & 0 
    \end{smallmatrix}\right)M' = \left(
    \begin{smallmatrix}
      0 & \alpha i \\
      \alpha i & 0 
    \end{smallmatrix}\right)$
  yields $z=0$.  Putting this into~(C1) we find that the second column
  of~$C$ equals $c^{-1}w$ times the first column of~$D$.

  Finally, we evaluate~(C3): $M\left(
    \begin{smallmatrix}
      0 & ui \\
      \gal{u}i & 0 
    \end{smallmatrix}\right)M' = \left(
    \begin{smallmatrix}
      0 & \alpha ui \\
      \alpha\gal{u} i & 0 
    \end{smallmatrix}\right)$.
  From $0=B\gal{u}iA'+AuiC'=-w\alpha i$ we infer $w=0$. Thus both $B$
  and $C$ are zero, and there remain the conditions $A(ui)D'=\alpha
  ui$ from~(C3) and $AiD'=\alpha i$ from~(C2).  Now
  $iD'i^{-1}=D^{-1}\det{D}=D^{-1}\alpha$ and the latter equality give
  $A=D$. The first equality then yields that $A$
  centralizes~$u$. But this means $A=\left(
    \begin{smallmatrix}
      1 & b \\
      0 & a 
    \end{smallmatrix}\right)\in\LL$,
  and $A=\left(
    \begin{smallmatrix}
      1 & 0 \\
      0 & 1 
    \end{smallmatrix}\right)$
  follows. 
\end{proof}

\begin{theo}\label{SigmaBetaHH}
  For\/ $\HH$, $\HH^\opp$ and\/ $W$ as in~\ref{def:WandHH} we have 
  $\Sigma_W=\HH^\times\HH^\opp$. 
\end{theo}
\begin{proof}
  We know from~\ref{quatInAnyChar} that the multiplicative group
  $\HH^\times\HH^\opp$ is contained in the stabilizer $(\GL4\KK)_W$.
  The subgroup $\set{\lambda_a\rho_a^{-1}}{a\in\HH^\times}$ of
  $\HH^\times \HH^\opp$ induces the full group $\SO{}{N|_{\Pu{\HH}}}$
  on~$W$ and $\GO{}{N|_{\Pu{\HH}}} =
  \KK^\times\,\SO{}{N|_{\Pu{\HH}}}$, cf.~\ref{GOnormApplied}.  %
  Thus it
  suffices to consider elements $M\in\GL4\KK$ that induce scalar
  multiples of the identity on~$W$.  Adapting $M$ by a further element
  of~$\HH^\opp$ (which induces a scalar multiple of the identity
  on~$W$ by~\ref{quatInAnyChar}) we may assume that~$M$
  fixes~$b_0$. Now the result follows
  from~\ref{onlyTheIdentityRemains}.
\end{proof}

It remains to understand the action of $\Sigma_W=\HH^\times\HH^\opp$
on $(\Sk4)/W$. Straightforward computations yield:

\begin{lemm}\label{orbitsWperp}
  If\/ $\LL/\KK$ is a separable extension then the action of\/
  the group $\HH^\times\HH^\opp$ on $W^\perp = \set{\left(
      \begin{smallmatrix}
        cxi & Y\delta \\
        Y\delta &  x i
      \end{smallmatrix}\right)}{x\in\KK,Y\in\LL}$
  is quasi-equivalent to that on~$W$; indeed 
  $G:=\left(
    \begin{smallmatrix}
      E_2 & 0 \\
      0 & \xi
    \end{smallmatrix}\right)\in\GL4\KK$
  satisfies $G\,\HH^\times G^{-1}=\HH^\opp$ and $GWG' = W^\perp$.
\qed
\end{lemm}

If $\Char\KK\ne2$ then~\ref{orbitsWperp} describes the action on
$(\Sk4)/W$ because $W^\perp$ is a vector space complement to~$W$
in~$\Sk4$. The remaining case $\Char\KK=2$ is more involved. We treat
the inseparable case, as well. 

\begin{lemm}\label{orbitsModWCharTwo}
  Assume $\Char\KK=2$ and $\Formspace1cdt\in\mathcal P_2$. We write
  $W:=W^t_{c,d}$ and\/ $\HH:=\HH_\KK^{-d,-c}$ if $t=1$
  and $\HH:=\KK(\sqrt{d},\sqrt{c})$
  otherwise. In any case, put\/ $\LL:=\KK[X]/(X^2+tX+d)$. The action
  $\omega\colon(\HH^\times\HH^\opp)\times(\Sk4)/W\to(\Sk4)/W$ can be
  described as follows. 
  \begin{enum}
  \item\label{orbitsModWCharTwoSep}%
    If\/ $t=1$ then the action~$\omega$ is equivalent to the action\/
    $\omega_1\colon(\HH^\times\HH^\opp)\times(\KK\times\LL)\to\KK\times\LL$
    given by\/
    \begin{align*}
      \omega_1\bigl(\lambda_{A+IB},(x,Y)\bigr) & =
      N(A+IB)\left(x,Y\right) \,, \\
      \omega_1\bigl(\rho_{C+ID},(x,Y)\bigr) & =
      \left(N(C+ID)x,CDu^{-1}x+C^2Y+(1+u^{-1})D^2\gal{Y}\right) \,.
    \end{align*}
  \item\label{orbitsModWCharTwoInsep}%
    If $t=0$ then $\HH$ is commutative, 
    $\HH^\times\HH^\opp=\HH^\times$ and the action~$\omega$ is
    equivalent to the action
    $\omega_0\colon\HH^\times\times(\KK\times\LL)\to\KK\times\LL$
    given by
    \begin{align*}
      \omega_0\bigl(\lambda_{A+IB},(x,Y)\bigr) & =
      (A+IB)^2(x,Y) \,.
    \end{align*}
  \end{enum}

\end{lemm}

\begin{proof}
  Assume $t=1$. Then $\set{\left(
      \begin{smallmatrix}
        xi & Y\delta \\
        Y\delta & 0
      \end{smallmatrix}\right)}{x\in\KK,Y\in\LL}$ is a vector space
  complement to~$W$ because the extension is separable.  Using
  $u=\left(
    \begin{smallmatrix}
      0 & -d \\
      1 & 1
    \end{smallmatrix}\right)$
  and the relations $\xi i=u^{-1}\delta$, $\xi\delta\xi'=(1+u^{-1})\delta$
  we compute 
  \[
  \left(
    \begin{matrix}
      A & c\gal{B} \\
      B & \gal{A}
    \end{matrix}\right)
  \left(\left(
      \begin{matrix}
        xi & Y\delta \\
        Y\delta & 0
      \end{matrix}\right) + W \right)
  \left(
    \begin{matrix}
      A' & B' \\
      c\gal{B}' & \gal{A}'
    \end{matrix}\right)
  = (A\gal{A}+cB\gal{B}) \left(
    \begin{matrix}
      xi & Y\delta \\
      Y\delta & 0
    \end{matrix}\right) + W
  \]
  and
  \begin{align*}& \left(
      \begin{matrix}
        C & cD\xi \\
        D\xi & C
      \end{matrix}\right)
    \left(\left(
        \begin{matrix}
          xi & Y\delta \\
          Y\delta & 0
        \end{matrix}\right) + W \right)
    \left(
      \begin{matrix}
        C' & (D\xi)' \\
        (D\xi)' & C'
      \end{matrix}\right)
    \\& = \left(
      \begin{matrix}
        (C\gal{C}+cD\gal{D})xi
        &
        \left(CD{u^{-1}}x+C^2Y+c(1+u^{-1})D^2\gal{Y}\right)\delta \\
        \left(CD{u^{-1}}x+C^2Y+c(1+u^{-1})D^2\gal{Y}\right)\delta
        & 0
      \end{matrix}\right) + W \,.
  \end{align*}
  Mapping $\left(
    \begin{smallmatrix}
      xi & Y\delta \\
      Y\delta & 0
    \end{smallmatrix}\right)+W$
  to $(x,Y)$ is a bijection from $(\Sk4)/W$ onto $\KK\times\LL$ that
  gives the equivalence to the action~$\omega_1$.

  Now assume $t=0$. Then $i\notin\LL$ and $\KK^{2\times2} =
  \LL\oplus\LL i$ yields that $\smallset{\left(
      \begin{smallmatrix}
        xi & Y \\
        Y' & 0 
      \end{smallmatrix}\right)}{x\in\KK,Y\in\LL}$ is a vector space
  complement to~$W$. Using $Z'\in Z+\KK i$ and $iZ'=Zi$ we compute
  \begin{align*}
  \left(
    \begin{matrix}
      A & c{B} \\
      B & {A}
    \end{matrix}\right)
  \left(
      \begin{matrix}
        xi & Y \\
        Y' & 0
      \end{matrix}\right) 
  \left(
    \begin{matrix}
      A' & B' \\
      c{B}' & {A}'
    \end{matrix}\right)
  &= \left(
    \begin{matrix}
      A^2xi+c(BY'A'+AYB') & ABxi + AYA'+cBY'B' \\
      ABxi + AY'A'+cBYB' & B^2xi + AY'B'+BYA'
    \end{matrix}\right) 
  \\
  &\in \left(
    \begin{matrix}
      (A^2+cB^2)xi + c\psi(A,B,Y) & (A^2+cB^2)Y \\
      \left((A^2+cB^2)Y\right)' & 0
    \end{matrix}\right) + W
\end{align*}
for $\psi(A,B,Y) := BY'A'+AYB'+AY'B'+BYA' = B(Y+Y')A'+A(Y+Y')B' =
ByiA'+AyiB' = 0 $. It remains to note $(A+IB)^2=A^2+I^2B^2=A^2+cB^2$. 
\end{proof}

In marked contrast to the case where $\Char\KK\ne2$
(cf.~\ref{orbitsWperp} and~\ref{quatInAnyChar}) the action $\omega_1$
in~\ref{orbitsModWCharTwo} is \emph{not} (quasi-) equivalent\footnote{
The situation is different in the inseparable case where $\omega_0 =
\omega_2$ because the multiplication is commutative and $N(a)=a^2$.} %
to the
action $\omega_2\colon(\HH^\times\HH^\opp)\times\Pu\HH\to\Pu\HH$ given
by $\omega_2\left((\lambda_a,\rho_b),x\right) = N(b)ax\tilde{a}$: the
subspace $\{0\}\times\LL$ is invariant under~$\omega_1$ but there is
no two-dimensional invariant subspace under~$\omega_2$.
However, we have the following. 

\begin{lemm}\label{compareConjModW}
  The restriction of $\omega_1$ to $\{0\}\times\LL$ is equivalent to
  the action induced by $\omega_2$ on the quotient $\Pu\HH/\KK$ modulo
  the subspace $\KK$ which is invariant under~$\omega_2$.

  The orbits in $\{0\}\times\LL$ are represented by a set
  $R_{W^\perp}\subseteq\Pu\HH$ such that %
  $\smallset{x^2}{x\in R_{W^\perp}}$ represents the orbits
  $v^2N(\HH^\sqt)+\KK^\sq$ of the group\/ $N(\HH^\sqt)$
  on~$\smallset{v^2+\KK^\sq}{v\in\Pu\KK}$.

  The orbits in $\KK^\times\times\LL$ are represented by
  $R_{\HH^\times}\times\{0\}$ where $R_{\HH^\times}$ is a set of
  representatives for the cosets in the multiplicative group
  $\KK^\times/N(\HH^\times)$.
\end{lemm}

\begin{proof}
  Applying the element $G$ from~\ref{orbitsWperp} to
  $(W+W^\perp)/(W\cap W^\perp)$ interchanges the two irreducible
  summands $W/(W\cap W^\perp)$ and $W^\perp/(W\cap W^\perp)$.

  From~\ref{conjugacyQuaternions} we infer that $x,v\in\Pu\HH$ (i.e.,
  both with trace~$0$) represent cosets $x+\KK$ and $v+\KK$ in the
  same $\omega_2$-orbit precisely if there exists $z\in\HH^\times$
  such that $\smallset{N(x+k)}{k\in\KK}=x^2+\KK^\sq$ has nonempty
  intersection with
  $\smallset{N(v+k)N(z^2)}{k\in\KK}=v^2N(z^2)+\KK^\sq$.

  For any $x\in\KK^\times$ the orbit
  $\smallset{(x,CDu^{-1}x)}{(C,D)\in\LL^2\setminus\{(0,0)\}}$ of
  $(x,0)$ meets each one of the cosets $\KK\times\{Y\}$. Therefore, it
  suffices to consider the action on the quotient modulo
  $\{0\}\times\LL$ to prove the last claim.
\end{proof}

\begin{theo}\label{omegaQuaternionField}
  The orbit of\/ $\ker\widehat{\beta}\in\mathcal P_2$ under $\GL4\KK$
  contains an element of the form $\Formspace 1cdt$ and then also
  $W:=W^t_{c,d}$ as in~\ref{def:WandHH}. We identify $\KK^4$ with 
  $\HH:=\HH_\KK^{-d,-c}$ if\/ $X^2+tX+d$ is separable and\/
  $\HH:=\KK(\sqrt{d},\sqrt{c})$ otherwise.
  \begin{enum}
  \item In any case the group\/ $\Sigma_W=\HH^\times\HH^\opp$ acts
    with two orbits on\/~$\HH$, represented by\/~$0$ and\/~$1$.
  \item If\/ $\Char\KK\ne2$ then the orbits of\/ $\Sigma_W$ on
    $(\Sk4)/W$ are represented by a set\/ $R_\HH\subseteq\Pu\HH$ such
    that for each coset of\/ $N(\HH^\times)/N(\HH^\sqt)$ there is
    exactly one element in~$R_\HH$.
  \item If\/ $\Char\KK=2$ and\/ $t=1$ then the action of\/ $\Sigma_W$
    is described
    in~\ref{orbitsModWCharTwo}.\ref{orbitsModWCharTwoSep}. %
    The orbits on $(\Sk4)/W$ are represented by
    $R_{\HH^\times}\cup R_{W^\perp}$ as in~\ref{compareConjModW}. 
  \item If\/ $\Char\KK=2$ and\/ $t=0$ then the orbits of\/ $\Sigma_W$
    on $(\Sk4)/W$ are represented by the elements of
    $\smallset{rZ}{r\in R_H\cup\{0\},Z\in R_W}$ %
    where $R_H$ and $R_W$ are sets of representatives for the cosets
    in $\HH^\times/\HH^\sqt$ and for the one-dimensional subspaces of
    $(\Sk4)/W$, respectively. Thus the number of orbits equals the
    cardinality of the (infinite) field\/~$\KK$.
  \end{enum}
\end{theo}
\begin{proof}
  According to~\cite[8.6]{MR2431124}
  (cf.~\ref{orbitsFourSpecial}.\ref{passingPlanesAndQuaternions}) we
  find an element of the form $\Formspace 1cdt$ the orbit of\/
  $\ker\widehat{\beta}$. The rest follows from~\ref{SigmaBetaHH},
  \ref{orbitsWperp}, \ref{knarrTrick}, \ref{orbitsModWCharTwo},
  and~\ref{compareConjModW}.
\end{proof}

\begin{rema}
  If $\Char\KK\ne2$ one can give a very nice description of an 
  alternating map $\beta$ with
  $\ker{\widehat\beta}=W^0_{c,d}$. Writing $\HH:=\HH_{\LL/\KK}^c$ for
  $\LL=\KK(\sqrt{-d})$ and $x\mapsto\tilde{x}$ for the standard
  involution on~$\HH$ we obtain the alternating map
  $\beta_\HH\colon\HH\times\HH\to\Pu\HH\colon (x,y)\mapsto
  \tilde{x}y-\tilde{y}x$. 
  Evaluating this map at pairs of the basis elements $b_0=-h_3$,
  $b_1=1$, $b_2=h_1$, $b_3=h_2$ one finds  
  $\ker\widehat{\beta_\HH}=\smallspnK{\s01+\s23,\s02+d\s13,\s03-c\s12}
  = \Formspace 1{-d}{-c}0$, and $\ker\widehat{\beta_\HH}$ lies in the
  orbit of $W^0_{c,d}$.
\end{rema}

For the classical quaternion field $\HH=\HH_{\CC/\RR}^{1}$ over the
field~$\RR$ the group $\Aut{\gheis\HH P{\beta_\HH}}$ acts with only
three orbits on~$\gheis\HH P{\beta_\HH}$. In that case the Heisenberg
algebra $\gheis\HH P{\beta_\HH}$ is \emph{almost homogeneous} (in the
sense of~\cite{MR1724629}, \cite{MR1774872} and~\cite{MR2003153} where
 $\gheis\HH P{\beta_\HH}$ occurs as~$\mathrm{H}^4_\HH$). %
For a general quaternion field, the group $\Sigma_{\beta_\HH}$ still
acts transitively but there may be more than two orbits
on~$P$, cf.~\ref{exQuaternions}. 

We will interpret our result~\ref{omegaQuaternionField} for several
cases explicitly in~\ref{exQuaternions} and~\ref{splitQuaternions}
below. We introduce some more notation (which appears to be quite
standard for quaternion algebras if $\Char\ne2$). 

\begin{nota}
  Let $\LL/\KK$ be a quadratic field extension where
  $\LL\cong\KK[X]/(X^2+tX+d)$ for some irreducible polynomial
  $X^2+tX+d\in\KK[X]$, and pick $c\in\KK^\times$. In order to indicate
  briefly the construction of~$\LL$ we will denote the quaternion
  algebra $\HH_{\LL/\KK}^{c\phantom{,}}$ also by~$\HH_\KK^{-d,-c}$.

  The $\KK$-algebra $\HH_\KK^{-d,-c}$ can also be described using a
  basis $h_0=1$, $h_1$, $h_2$, $h_3=h_2h_1$ where $h_1\in\LL$ is a
  root of $X^2+tX+d$ and $h_2\in\LL^\perp$ is a root of $X^2+c$; then
  $h_1h_2=th_2-h_3$. This description does not depend on the fact that
  $\LL=\KK[X]/(X^2+tX+d)$ is a field; we will use the notation
  $\HH_\KK^{-d,-c}$ for any pair $(d,c)\in(\KK\setminus\{0\})^2$ to
  denote a $4$-dimensional associative $\KK$-algebra with a basis
  $h_0^{}=1$, $h_1^{}$, $h_2^{}$, $h_3^{}$ satisfying
  $h_1^2=-th_1^{}-d$, $h_2^2=-c$, $h_2^{}h_1^{}=h_3^{}$,
  $h_1^{}h_2^{}=th_2^{}-h_3^{}$ for $t=0$ if $\Char\KK\ne2$ and $t=1$
  if $\Char\KK=2$.

  Such a quaternion algebra is a quaternion field precisely if its
  norm form is anisotropic; in all other cases it is isomorphic to
  $\HH_\KK^{1,1}$.
\end{nota}

\begin{exas}
  For a quaternion algebra $\HH=\HH_\KK^{-d,-c}$ the set $\HH^\sqt$
  need not be closed under multiplication (the interesting set for us
  is indeed $N(\HH^\sqt)=\smallset{N(x)^2}{x\in\HH^\times}$ which is
  closed under multiplication). For example, we
  have $-2=h_1^2$ and $-3=h_2^2$ in $(\HH_\QQ^{-1,-1})^\sqt$ but
  their product $6=(-2)(-3)$ does not belong to $(\HH_\QQ^{-1,-1})^\sqt$ because
  $\smallset{x^2}{x\in\Pu{\HH_\QQ^{-1,-1}}}\subseteq\smallset{q\in\QQ}{q\le0}$
  and
  $\smallset{x\in\HH_\QQ^{-1,-1}}{x^2\in\QQ}= \QQ\cup\Pu{\HH_\QQ^{-1,-1}}$
  while $6\notin\QQ^\sqt$.

  As a second interesting example we mention
  $\RR^{2\times2}\cong\HH_\RR^{1,1}$ (see~\ref{splitQuaternions}
  below); here $\left(
    \begin{smallmatrix}
      -1 & 2 \\
      0 & -1
    \end{smallmatrix}\right)$
  is not a square but $\left(
    \begin{smallmatrix}
      -1 & 2 \\
      0 & -1
    \end{smallmatrix}\right) = \left(
    \begin{smallmatrix}
      0 & -1 \\
      1 & 0
    \end{smallmatrix}\right)^2
  \left(
    \begin{smallmatrix}
      1 & -1 \\
      0 & 1
    \end{smallmatrix}\right)^2$. 
\end{exas}

\begin{exas}\label{exQuaternions}
  We consider $\HH=\HH_\KK^{-d,-c}$ and $N(\HH^\times)/N(\HH^\sqt)$ for
  different explicit choices of~$\KK$ and~$(-d,-c)$: 
  \begin{enum}
  \item For $\KK=\RR$ and $(-d,-c)=(-1,-1)$ we find
    $N(\HH^\times)=\set{r\in\RR}{r>0} = N(\HH^\sqt)$ and
    $\left|N(\HH^\times)/N(\HH^\sqt)\right|=1$.
  \item For $\KK=\QQ$ and $(-d,-c)=(-1,-1)$ we have
    $N(\HH^\times)=\set{r\in\QQ}{r>0}$ by Lagrange's four-square
    theorem (cf.~\cite[II.8.3]{MR0506372}). The group $\QQ^\times$ is
    isomorphic to a direct sum of $\ZZ/2\ZZ$ and countably many
    infinite cyclic groups. The norm group $N(\HH^\times)$ is the
    unique subgroup of index~$2$ and the quotient
    $N(\HH^\times)/N(\HH^\sqt)$ is an elementary abelian group of
    countably infinite rank.
  \item According to Fermat's theorem on sums of two squares
    (cf.~\cite[II.8.1]{MR0506372}) an odd prime number~$p$ is the sum
    of two squares of integers precisely if $p\equiv1\pmod{4}$. For
    $\KK=\QQ$ and $(-d,-c)=(-1,-c)$ with any $c\in\QQ^\times$ this
    implies that the group $N(\HH^\times)$ contains infinitely many
    primes. Thus $N(\HH^\times)/N(\HH^\sqt)$ is countably infinite in
    these cases, as well.
  \item Let $p>2$ be a prime number, and let $\KK$ be a finite
    extension of the field $\QQ_p$ of $p$-adic numbers. Then there
    exists, up to isomorphism, precisely one quaternion field~$\HH$
    over~$\KK$, and $\KK^\times=N(\HH^\times)$,
    cf.~\cite[VI\,2.10]{MR2104929}. There are four square classes
    in~$\KK^\times$ (see~\cite[VI,\,2.22]{MR2104929}), and
    $N(\HH^\times)/N(\HH^\sqt)=\KK^\times/\KK^\sqt\cong(\ZZ/2\ZZ)^2$.
  \item Now let $\KK$ be a finite extension of the field $\QQ_2$ of
    $2$-adic numbers, of degree~$e$. Again, there exists precisely one
    quaternion field~$\HH$ over~$\KK$ (up to isomorphism), and
    $\KK^\times=N(\HH^\times)$. However, there are $2^{e+2}$ square classes
    in~$\KK^\times$, and
    $N(\HH^\times)/N(\HH^\sqt)=\KK^\times/\KK^\sqt\cong(\ZZ/2\ZZ)^{e+2}$,
    cf.~\cite[VI\,2.23]{MR2104929}.
  \end{enum}
\end{exas}

\begin{ndef}[Split Quaternions in Odd Characteristic]\label{splitQuaternions}
  Let $\KK$ be any field with $\Char\KK\ne2$.
  It is well known that then the quaternion algebra $\HH_\KK^{1,1}$
  splits, and is therefore isomorphic to the algebra
  $\KK^{2\times2}$; the multiplicative form is then $N(x)=\det(x)$.
  See~\cite[8.7]{MR2431124}, where an explicit
  isomorphism is given to show that $\ker\widehat{\beta_{\KK^{2\times2}}}$ 
  belongs to the orbit of $T+S$ under the group~$\GL4\KK$.
  The automorphisms can be read off from~\ref{SigmaBetaHH} or from~\ref{PhiDelta}.
  Note that the multiplicative group of the split quaternion algebra
  is (isomorphic to)~$\GL2\KK$; the subgroup $\SL2\KK$ induces a group
  isomorphic to $\PSL2\KK$ on~$P$, acting as the group of proper
  hyperbolic motions with respect to the form~$N|_P$.
\end{ndef}

We translate our result~\ref{PhiDelta} into the description using
split quaternions:  

\begin{theo}\label{orbitsSplitQuaternions}
  If $\Char\KK\ne2$ then there are three orbits of
  $\Sigma_{\beta_{\KK^{2\times2}}}$ on $\KK^{2\times2}$; characterized
  by the rank of their members.  Representatives for the orbits are
  thus $\left(
    \begin{smallmatrix}
      0 & 0 \\
      0 & 0 
    \end{smallmatrix}\right)$, $\left(
    \begin{smallmatrix}
      1 & 0 \\
      0 & 0 
    \end{smallmatrix}\right)$, and $\left(
    \begin{smallmatrix}
      1 & 0 \\
      0 & 1 
    \end{smallmatrix}\right)$. 

  Each orbit on~$P$ (i.e., on the set of matrices with vanishing
  trace) is obtained by fusion of a conjugacy class with all its
  images under multiplication with scalars. These orbits are
  represented by the elements of the set
  $\{0\} \cup \set{\left(
      \begin{smallmatrix}
      0 & -d \\
      1 & 0
    \end{smallmatrix}\right)}{d\in R_*\cup\{0\}}$
  where $R_*$ is, again, a set of representatives for the cosets in
  $\KK^\times/\KK^\sqt$. 
\end{theo}

\begin{proof}
  The orbits on $\KK^{2\times2}$ are clear from~\ref{SigmaBetaHH} or
  from~\ref{PhiDelta}.  Those on $P$ can be read off
  from~\ref{SigmaBetaHH} and~\ref{conjugacyQuaternions}: we have
  $N(x)=\det x$ and the elements of~$P$ are those with trace~$0$.
  Thus elements of $P$ are conjugates if, and only if, they have the
  same norm, and belong to the same orbit under
  $\Sigma_{\beta_{\KK^{2\times2}}}$ if their determinants differ by a
  square in~$\KK^\times$.
  
  Now $\Char\KK\ne2$ yields that $0$ is the only scalar multiple of
  the identity matrix in~$P$. For every other element $x\in P$ we may
  choose a representative for the orbit of~$x$ in Frobenius normal
  form $X_d:=\left(
    \begin{smallmatrix}
      0 & -d \\
      1 & 0
    \end{smallmatrix}\right)$, where $d\in R_*$ is the representative
  of~$\det(x)\KK^\sq$. 
\end{proof}

\bigbreak
\section{Results}\label{sec:results}

Let $H:=\gheis{V}{Z}{\beta}$ be a reduced Heisenberg algebra with
$\dim V=4$. From~\ref{isoHeis} we know that the orbits of
$\Aut{\gheis{V}{Z}{\beta}}$ are controlled by the orbits of
$\Sigma_\beta$ on $V$ and on~$Z$, respectively.  In particular, we
find $\omega(\gheis{V}{Z}{\beta}) = \omega_V+\omega_Z+1$ where
$\omega_V$ denotes the number of orbits in $H\setminus Z$ and
$\omega_Z$ the number of those in $Z\setminus\{0\}$.  The numbers
$\omega_V$ and $\omega_Z$ can be read off from our discussion of the
possible cases for $\ker{\hat\beta}$ in
Sections~\ref{sec:directComputations}, \ref{sec:fieldextensions}
and~\ref{sec:quaternions} above. %
Table~\ref{tab:results} collects these results; the column
``reference'' indicates the place where the corresponding result is
proved.

\newpage
\begin{table}[htb]
\begin{ndef}[Numbers of Orbits under Automorphisms of Reduced Heisenberg Algebras]
\label{tab:results}
\quad\\[1ex]
  \(
\renewcommand{\arraystretch}{1.14}%
\begin{array}{|c|c|c||c|l|}
  \hline
  \ker\widehat\beta & \omega_V & \omega_Z  & \omega & \text{References} \\
  \hline
  \hline
  \smallspnK{\s01} & 2 & 3 & 6 & \ref{stabS01}\\
  \hline
  \smallspnKperp{\s01+\s23} & 1 & 1 & 3 & \ref{stabGSp}\\
  \smallspnK{\s01+\s23}, \Char\KK\ne2 & 1 & 1+|R_*| & 3+|R_*| & \\
  \smallspnK{\s01+\s23}, \Char\KK=2 & 1 & 2+|R_+|+|R_\wp| & 4+|R_+|+|R_\wp| & \\
  \hline
  E & 3 & 3 & 7 & \ref{stabE} \\
  \hline
  T & 2 & 3 & 6 & \ref{stabT} \\
  T^\perp & 2 & 2 & 5 & \\
  \hline
  S & 2 & 2 & 5 & \ref{stabS} \\
  S^\perp & 2 & 2 & 5 & \\
  \hline
  J(F) & 2 & 1 & 4 & \ref{stabJF} \\
  \hline
  E+T & 3 & 2 & 6 & \ref{stabET} \\
  \hline
  E+S & 4 & 3 & 8 & \ref{stabES} \\
  \hline
  T+S, \Char\KK\ne2 & 2 & 1+|R_*|         & 4+|R_*| & \ref{stabTS}, \ref{SigmaBetaHH} \\
  T+S, \Char\KK=2   & 2 & 1+|R_*|+|R_\wp| & 4+|R_*|+|R_\wp| & 
  \\
  \hline
   P_\LL^\perp\in\mathcal{P}_1^\perp   & 1 & 1 & 3 & \ref{complexGHaut} \\
   P_\LL^{}\in\mathcal{P}_1         , \LL/\KK\text{ sep.} & 1 &
   \mathrm{HF}(\LL/\KK) & 2+\mathrm{HF}(\LL/\KK) &  \\
   P_\LL^{}\in\mathcal{P}_1         , \LL/\KK\text{ insep.} & 1 &
   \begin{array}[t]{r}
   |R_{\KK/\LL^\sq}^{}| + |R^{(2)}_{\KK/\LL^\sq}| \ \ \\ {} + |R_{\actSL}| -1       
   \end{array} &
   \begin{array}[t]{r}
     |R_{\KK/\LL^\sq}^{}| + |R^{(2)}_{\KK/\LL^\sq}| \ \ \\ {} + |R_{\actSL}| +1
   \end{array} & \\
  \hline
  W^0_{c,d}\in\mathcal{P}_2, \Char\KK\ne2 & 1 & \left|N(\HH^\times)/N(\HH^\sqt)\right| & 2+\left|N(\HH^\times)/N(\HH^\sqt)\right| &
  \ref{omegaQuaternionField} \\
  W^1_{c,d}\in\mathcal{P}_2, \Char\KK=2   & 1 & |R_{\HH^\times}|+|R_{W^\perp}| & 2+|R_{\HH^\times}|+|R_{W^\perp}| &  \\
  W^0_{c,d}\in\mathcal{P}_2, \Char\KK=2   & 1 & |R_H|+|R_W| & 2+|R_H|+|R_W|=|\KK| &  \\
  \hline
   P_\LL^0\in\mathcal{P}_3  & 2 & 1 + \left|R_{N}\right| & 4+ \left|R_{N}\right| & \ref{complexGHaut} \\
  \hline
\end{array}
\)
\end{ndef}
\end{table}

\bigbreak
\begin{ndef}[Notation]
In Table~\ref{tab:results} we use the following notation.
\begin{description*}
\item 
  [$R_*$] denotes a set of coset representatives for the classes in
  $\KK^\times/\KK^\sqt$.
\item
  [$R_+$] is a set of representatives for the orbits of $\KK^\sqt$
  on the additive group $\KK/\KK^\sq$.
\item
  [$R_\wp$] is (if $\Char\KK=2$) a set of coset representatives for the additive group
  $\KK/\wp$ where $\wp:=\set{x+x^2}{x\in\KK}$, cf.~\ref{Arf}.
\item
  [$\mathrm{HF}(\LL/\KK)$] denotes the number of equivalence classes of non-zero
  hermitian forms on~$\LL^2$, cf.~\ref{HermiteEq}.
\item
  [$R_{\KK/\LL^\sq}^{}$ and $R^{(2)}_{\KK/\LL^\sq}$] represent the two
  kinds of orbits under the action of $\GL2{\LL}$ on~$\KK^2$,
  see~\ref{diagonalCharTwo}. 
\item
  [$R_{\actSL}$] represents the orbits under the action $\actSL$ of
  $\SL2\LL$ on~$\KK^2$, see~\ref{complexGHaut}.
\item
  [$R_{\HH^\times}$] is a set of representatives for
  $\KK^\times/N(\HH^\times)$ where $\HH=\HH_{\KK}^{-d,-c}$. 
\item
  [$R_{W^\perp}\subseteq\Pu\HH$] is (for a quaternion field $\HH$
  over~$\KK$ with $\Char\KK=2$) a set such that the set %
  $\smallset{x^2}{x\in R_{W^\perp}}$ represents the orbits
  $v^2N(\HH^\sqt)+\KK^\sq$ of\/ $N(\HH^\sqt)$
  on~$\smallset{v^2+\KK^\sq}{v\in\Pu\KK}$.
\item
  [$R_H$] is (if $\HH=\KK(\sqrt{d},\sqrt{c})$ is a
  purely inseparable extension of degree~$4$) a set of coset
  representatives for $\HH^\times/\HH^\sqt$. 
\item
  [$R_W$] is a set of
  representatives for the one-dimensional subspaces of $(\Sk4)/W$,
  see~\ref{omegaQuaternionField}.
\item
  [$R_{N}^{}$] is a set of coset representatives for
  the multiplicative group 
  $\KK^\times/(N_{\LL/\KK}(\LL^\times)\left<-1\right>)$, 
  cf.~\ref{complexGHaut}.
\end{description*}  
\end{ndef}

\goodbreak
\section{Open problems}

\begin{prob}\label{prob:fullAut}
  How is $\Aut{\gheis[\ZZ]VZ\beta}$ related to
  $\Aut{\gheis VZ\beta}$, or, more generally, how is the automorphism
  group $\Aut{\mathfrak n}$ of a nilpotent Lie algebra $\mathfrak n$ over~$\KK$
  related to the automorphism group $\Aut[\ZZ]{\mathfrak n}$ of the Lie ring?
  Definitely, the subgroup
  \[
  H:=\set{(v,z,a)\mapsto(v,z+\tau(v),\alpha(a)+\zeta(v,z))}%
  {\begin{array}{cc}
      \tau\in\Hom VZ,\alpha\in\GL{}A,\\
      \zeta\in\Hom{V\times Z}{A}
   \end{array}}
  \]
  from~\ref{isoHeis} has to be enlarged to
  \[
  H_\ZZ:=\set{(v,z,a)\mapsto(v,z+\tau(v),\alpha(a)+\zeta(v,z))}%
  {\begin{array}{cc}
      \tau\in\Hom[\ZZ]VZ,\alpha\in\Aut[\ZZ]A,\\
      \zeta\in\Hom[\ZZ]{V\times Z}{A}
   \end{array}}
  \,;
  \]
  we must use arbitrary
  additive homomorphisms instead of $\KK$-linear ones. 
  Is it true that
  $\Aut{\gheis[\ZZ]VZ\beta}$ equals
  $\Aut\KK\ltimes\left(\Aut{\gheis VZ\beta}\cdot H_\ZZ\right)$~?

  A positive partial answer is known: according
  to~\cite[{1.1.1}]{MR2410562}, the assertion is true for $\gheis
  VZ\beta=\gheis{\KK^2}\KK\det$.
\end{prob}

\begin{prob}
  A deeper understanding of the action
  $\actSL\colon \SL2\LL\times\KK^2 \to \KK^2$
  in~\ref{complexGHaut} would be very welcome. 
\end{prob}

\begin{prob}
  Our results allow to identify the cases where
  $\omega(\gheis{V}{Z}{\beta})$ is finite (depending on the ground
  field~$\KK$), and then those where the invariant~$\omega$ takes on
  values that are particularly small. It appears that these algebras
  merit deeper study.
\end{prob}

\begin{acks}
  The authors owe Norbert Knarr at least one cup of coffee for
  the simple computational arguments in~\ref{knarrTrick} (replacing
  much deeper arguments involving Witt cancellation and the
  Skolem--Noether Theorem).

  The second author was supported by SFB 478 ``Geometrische Strukturen
  in der Mathematik'', M\"unster, Germany; a substantial part of this
  paper was written during a stay at M\"unster in 2009.   
\end{acks}

\bigbreak

\begin{thebibliography}{10}
\providecommand{\url}[1]{\href{#1}{#1}}
\providecommand{\urlprefix}{}
  \providecommand{\doi}[1]{\href{http://dx.doi.org/#1}{\tt doi:#1}}
\providecommand{\MR}[1]{\relax\ifhmode\unskip\space\fi \MRnumberextract#1 \,}
\def\MRnumberextract#1 #2\,{\MRhref{#1}{#2}}%
\providecommand{\MRhref}[2]{%
  \href{http://www.ams.org/mathscinet-getitem?mr=#1}{MR\,#1 #2}%
}
\providecommand{\ZBL}[1]{\relax\ifhmode\unskip\space\fi \ZBLhref{#1}}
\providecommand{\ZBLhref}[1]{%
  \href{http://www.zentralblatt-math.org/NEW/zmath/en/search/?q=an:#1&format=c%
omplete}{Zbl #1}
}
\providecommand{\JfM}[1]{\relax\ifhmode\unskip\space\fi \JfMhref{#1}}
\providecommand{\JfMhref}[1]{%
  \href{http://www.zentralblatt-math.org/NEW/zmath/en/search/?q=an:#1}{JfM #1}
}
\providecommand{\href}[2]{#2}
\providecommand{\bbldiplomarbeit}{Diplomarbeit}
 %
\def\bbland{and}                \def\bbletal{et~al.}
\def\bbleditors{editors}        \def\bbleds{eds.}
\def\bbleditor{editor}          \def\bbled{ed.}
\def\bbledby{edited by}
\def\bbledition{edition}        \def\bbledn{edn.}
\def\bblvolume{volume}          \def\bblvol{vol.}
\def\bblof{of}
\def\bblnumber{number}          \def\bblno{no.}
\def\bblin{in}
\def\bblpages{pages}            \def\bblpp{pp.}
\def\bblpage{page}              \def\bblp{p.}
\def\bbleidpp{pages}
\def\bblchapter{chapter}        \def\bblchap{chap.}
\def\bbltechreport{Technical Report}
\def\bbltechrep{Tech. Rep.}
\def\bblmthesis{Master's thesis}
\def\bblphdthesis{Ph.D. thesis}
\def\bblfirst{First}            \def\bblfirsto{1st}
\def\bblsecond{Second}          \def\bblsecondo{2nd}
\def\bblthird{Third}            \def\bblthirdo{3rd}
\def\bblfourth{Fourth}          \def\bblfourtho{4th}
\def\bblfifth{Fifth}            \def\bblfiftho{5th}
\def\bblst{st}  \def\bblnd{nd}  \def\bblrd{rd}
\def\bblth{th}
\def\bbljan{January}  \def\bblfeb{February}  \def\bblmar{March}
\def\bblapr{April}    \def\bblmay{May}       \def\bbljun{June}
\def\bbljul{July}     \def\bblaug{August}    \def\bblsep{September}
\def\bbloct{October}  \def\bblnov{November}  \def\bbldec{December}
 

\newcommand{\Capitalize}[1]{\uppercase{#1}}
\newcommand{\capitalize}[1]{\expandafter\Capitalize#1}

\bibitem{MR0082463}
E.~Artin, \emph{Geometric algebra}, Interscience Publishers, Inc., New
  York-London, 1957. \MR{0082463 (18,553e)} \ZBL{0077.02101}

\bibitem{MR2161930}
G.~Belitskii, R.~Lipyanski, \bbland{} V.~Sergeichuk, \emph{Problems of
  classifying associative or {L}ie algebras and triples of symmetric or
  skew-symmetric matrices are wild}, Linear Algebra Appl. \textbf{407} (2005),
  249--262, ISSN 0024-3795, \doi{10.1016/j.laa.2005.05.007}. \MR{2161930
  (2006i:17014)} \ZBL{1159.17304}

\bibitem{MR0107661}
N.~Bourbaki, \emph{\'{E}l\'ements de math\'ematique. {P}remi\`ere partie: {L}es
  structures fondamentales de l'analyse. {L}ivre {II}: {A}lg\`ebre. {C}hapitre
  9: {F}ormes sesquilin\'eaires et formes quadratiques}, Actualit\'es Sci. Ind.
  no. 1272, Hermann, Paris, 1959. \MR{0107661 (21 \#6384)} \ZBL{0102.25503}

\bibitem{MR1935285}
P.~M. Cohn, \emph{Basic algebra}, Springer-Verlag London Ltd., London, 2003,
  ISBN 1-85233-587-4. \MR{1935285 (2003m:00001)} \ZBL{1003.00001}

\bibitem{MR0072144}
J.~Dieudonn{\'e}, \emph{La g\'eom\'etrie des groupes classiques}, Ergebnisse
  der {M}athematik und ihrer {G}renzgebiete ({N}.{F}.) ~5, Springer-Verlag,
  Berlin, 1955. \MR{0072144 (17,236a)} \ZBL{0221.20056}

\bibitem{MR2410562}
T.~Grundh{\"o}fer \bbland{} M.~Stroppel, \emph{Automorphisms of {V}erardi
  groups: small upper triangular matrices over rings}, Beitr\"age Algebra Geom.
  \textbf{49} (2008), \bblno{}~1, 1--31, ISSN 0138-4821,
  \urlprefix\url{http://www.emis.de/journals/BAG/vol.49/no.1/1.html}.
  \MR{2410562 (2009d:20079)} \ZBL{05241751}

\bibitem{MR2003153}
J.~Hoheisel \bbland{} M.~Stroppel, \emph{More about embeddings of almost
  homogeneous {H}eisenberg groups}, J. Lie Theory \textbf{13} (2003),
  \bblno{}~2, 443--455, ISSN 0949-5932,
  \urlprefix\url{http://www.emis.de/journals/JLT/vol.13_no.2/14.html}.
  \MR{2003153 (2004g:22007)} \ZBL{1030.22002}

\bibitem{MR650245}
B.~Huppert \bbland{} N.~Blackburn, \emph{Finite groups. {II}}, Grundlehren der
  Mathematischen Wissenschaften  242, Springer-Verlag, Berlin, 1982, ISBN
  3-540-10632-4. \MR{650245 (84i:20001a)} \ZBL{0477.20001}

\bibitem{MR1009787}
N.~Jacobson, \emph{Basic algebra. {II}}, W. H. Freeman and Company, New York,
  \bblsecondo{} \bbledn{}, 1989, ISBN 0-7167-1933-9. \MR{1009787 (90m:00007)}
  \ZBL{0694.16001}

\bibitem{MR2104929}
T.~Y. Lam, \emph{Introduction to quadratic forms over fields}, Graduate Studies
  in Mathematics ~67, American Mathematical Society, Providence, RI, 2005, ISBN
  0-8218-1095-2. \MR{2104929 (2005h:11075)} \ZBL{1068.11023}

\bibitem{MR1878556}
S.~Lang, \emph{Algebra}, Graduate Texts in Mathematics  211, Springer-Verlag,
  New York, \bblthirdo{} \bbledn{}, 2002, ISBN 0-387-95385-X. \MR{1878556
  (2003e:00003)} \ZBL{0984.00001}

\bibitem{MR1681303}
H.~M{\"a}urer, \emph{Die {Q}uaternionenschiefk\"orper}, Math. Semesterber.
  \textbf{46} (1999), \bblno{}~1, 93--96, ISSN 0720-728X,
  \doi{10.1007/s005910050055}. \MR{1681303 (2000a:12001)} \ZBL{0937.11055}

\bibitem{MR1484565}
H.~M{\"a}urer \bbland{} M.~Stroppel, \emph{Groups that are almost homogeneous},
  Geom. Dedicata \textbf{68} (1997), \bblno{}~2, 229--243, ISSN 0046-5755,
  \doi{10.1023/A:1005090519480}. \MR{1484565 (98k:20053)} \ZBL{0890.20025}

\bibitem{MR0506372}
J.~Milnor \bbland{} D.~Husemoller, \emph{Symmetric bilinear forms}, Ergebnisse
  der {M}athematik und ihrer {G}renzgebiete ~73, Springer-Verlag, New York,
  1973. \MR{0506372 (58 \#22129)} \ZBL{0292.10016}

\bibitem{MR770063}
W.~Scharlau, \emph{Quadratic and {H}ermitian forms}, Grundlehren der
  Mathematischen Wissenschaften  270, Springer-Verlag, Berlin, 1985, ISBN
  3-540-13724-6. \MR{770063 (86k:11022)} \ZBL{0584.10010}

\bibitem{MR0491938}
V.~V. Serge{\u\i}{\v{c}}uk, \emph{The classification of metabelian
  {$p$}-groups}, \bblin{} \emph{Matrix problems (Russian)}, \bblpp{} 150--161,
  Akad. Nauk Ukrain. SSR Inst. Mat., Kiev, 1977. \MR{0491938 (58 \#11109)}
  \ZBL{0444.20018}

\bibitem{MR1724629}
M.~Stroppel, \emph{Homogeneous symplectic maps and almost homogeneous
  {H}eisenberg groups}, Forum Math. \textbf{11} (1999), \bblno{}~6, 659--672,
  ISSN 0933-7741, \doi{10.1515/form.1999.018}. \MR{1724629 (2000j:22006)}
  \ZBL{0928.22008}

\bibitem{MR1774872}
M.~Stroppel, \emph{Embeddings of almost homogeneous {H}eisenberg groups}, J.
  Lie Theory \textbf{10} (2000), \bblno{}~2, 443--453, ISSN 0949-5932,
  \urlprefix\url{http://www.emis.de/journals/JLT/vol.10_no.2/14.html}.
  \MR{1774872 (2001g:22013)} \ZBL{0955.22009}

\bibitem{MR2431124}
M.~Stroppel, \emph{The {K}lein quadric and the classification of nilpotent
  {L}ie algebras of class two}, J. Lie Theory \textbf{18} (2008), \bblno{}~2,
  391--411, ISSN 0949-5932,
  \urlprefix\url{http://www.heldermann-verlag.de/jlt/jlt18/strola2e.pdf}.
  \MR{2431124 (2009e:17016)} \ZBL{1179.17013}

\bibitem{MR1189139}
D.~E. Taylor, \emph{The geometry of the classical groups}, Sigma Series in Pure
  Mathematics ~9, Heldermann Verlag, Berlin, 1992, ISBN 3-88538-009-9.
  \MR{1189139 (94d:20028)} \ZBL{0767.20001}

\bibitem{MR593560}
A.~L. Vishnevetski{\u\i}, \emph{Groups of class {$2$} and exponent {$p$} with
  commutant of order {$p\sp{2}$}}, Dokl. Akad. Nauk Ukrain. SSR Ser. A  (1980),
  \bblno{}~9, 9--11, 103, ISSN 0201-8446. \MR{593560 (82d:20026)}
  \ZBL{0453.20014}

\bibitem{MR795568}
A.~L. Vishnevetski{\u\i}, \emph{A system of invariants of certain groups of
  class {$2$} with commutator subgroup of rank two}, Ukrain. Mat. Zh.
  \textbf{37} (1985), \bblno{}~3, 294--300, 403, ISSN 0041-6053,
  \doi{10.1007/BF01059600}. \MR{795568 (86k:20033)} \ZBL{0604.20022}

\end{thebibliography}

\raggedright
\def\cprime{$'$}

\end{document}